\documentclass[letterpaper,12pt]{article}
\usepackage{subfig,amsmath,amssymb,amsthm,graphicx,tikz,bbm}
\usepackage[authoryear,longnamesfirst]{natbib}
\newtheorem{lemma}{Lemma}
\newtheorem{theorem}[lemma]{Theorem}
\newtheorem{assumption}[lemma]{Assumption}
\DeclareMathOperator*{\argmax}{argmax}

\usepackage{geometry,setspace}
\begin{document}
\title{Heterogeneous Treatment Effects with Mismeasured Endogenous Treatment\thanks{First version: November, 2015. I would like to thank Federico A. Bugni, V. Joseph Hotz, Shakeeb Khan, and Matthew A. Masten for their guidance and encouragement. I am also grateful to Luis E. Candelaria, Xian Jiang, Marc Henry, Ju Hyun Kim, Arthur Lewbel, Jia Li, Arnaud Maurel, Marjorie B. McElroy, Ismael Mourifi\'e, Naoki Wakamori, Yichong Zhang, and seminar participants at University of Tokyo, Singapore Management University, University of Oslo, UC-Irvine, UC-Davis, University of Western Ontario, University of Warwick, Duke, European Meeting of the Econometric Society, Asian Meeting of the Econometric Society, North American Summer Meeting of the Econometric Society, and Triangle Econometrics Conference.}}
\author{Takuya Ura\thanks{Department of Economics, University of California, Davis, One Shields Avenue, Davis, CA 95616-5270; Email: takura@ucdavis.edu}}
\maketitle
\vspace{-.3in}
\begin{abstract}
This paper studies the identifying power of an instrumental variable in the nonparametric heterogeneous treatment effect framework when a binary treatment is mismeasured and endogenous. 
Using a binary instrumental variable, I characterize the sharp identified set for the local average treatment effect under the exclusion restriction of an instrument and the deterministic monotonicity of the true treatment in the instrument. 
Even allowing for general measurement error (e.g., the measurement error is endogenous), it is still possible to obtain finite bounds on the local average treatment effect. 
Notably, the Wald estimand is an upper bound on the local average treatment effect, but it is not the sharp bound in general. 
I also provide a confidence interval for the local average treatment effect with uniformly asymptotically valid size control.  
Furthermore, I demonstrate that the identification strategy of this paper offers a new use of repeated measurements for tightening the identified set.
\begin{description}\item Keywords: Local average treatment effect; Instrumental variable; Nonclassical measurement error; Endogenous measurement error; Partial identification\end{description}
\end{abstract}
\newpage 

\section{Introduction}\label{sec1}
Treatment effect analyses often entail a measurement error problem as well as an endogeneity problem. 
For example, \cite{black/sanders/taylor:2003} document a substantial measurement error in educational attainments in the 1990 U.S. Census.
At the same time, educational attainments are endogenous treatment variables in a return to schooling analysis, because unobserved individual ability affects both schooling decisions and wages \citep{card:2001}.  
The econometric literature, however, has offered only a few solutions for addressing the two problems at the same time. 
An instrumental variable is a standard technique for correcting endogeneity and measurement error \citep[e.g.,][]{angrist/krueger:2001}, but, to the best of my knowledge, no existing research has explicitly investigated the identifying power of an instrumental variable for the heterogeneous treatment effect when the treatment is both mismeasured and endogenous.\footnote{Many existing methods, including \cite{mahajan:2006} and \cite{lewbel:2007}, allow for the treatment effect to be heterogeneous due to observed variables. In this paper I focus on the heterogeneity due to unobserved variables by considering the local average treatment effect framework.}  

I consider a mismeasured treatment in the framework of \cite{imbens/angrist:1994} and \cite{angrist/imbens/rubin:1996}, and focus on the local average treatment effect as a parameter of interest. 
My analysis studies the identifying power of a binary instrumental variable under the following two assumptions: (i) the instrument affects the outcome and the measured treatment only through the true treatment (the exclusion restriction of an instrument),  and (ii) the instrument weakly increases the true treatment (the deterministic monotonicity of the true treatment in the instrument). 
These assumptions are an extension of \cite{imbens/angrist:1994} and \cite{angrist/imbens/rubin:1996} into the framework with mismeasured treatment.
The local average treatment effect is the average treatment effect for the compliers, that is, the subpopulation whose true treatment status is strictly affected by an instrument.
Focusing on the local average treatment effect is meaningful for a few reasons.\footnote{\cite{deaton:2009} and \cite{heckman/urzua:2010} are cautious about interpreting the local average treatment effect as a parameter of interest. See also \cite{imbens:2010,imbens:2014} for a discussion.}  
First, the local average treatment effect has been a widely used parameter to investigate the heterogeneous treatment effect with endogeneity.   
My analysis offers a tool for a robustness check to those who have already investigated the local average treatment effect.
Second, the local average treatment effect can be used to extrapolate to the average treatment effect or other parameters of interest. 
\cite{imbens:2010} emphasize the utility of reporting the local average treatment effect in addition to the other parameters of interest, because the extrapolation often requires additional assumptions and can be less credible than the local average treatment effect.

The mismeasured treatment prevents the local average treatment effect from being point-identified. 
As in \cite{imbens/angrist:1994} and \cite{angrist/imbens/rubin:1996}, the local average treatment effect is the ratio of the intent-to-treat effect over the size of compliers.\footnote{The intent-to-treat effect is defined as the mean difference of the outcome between the two groups defined by the instrument. The size of compliers is the probability of being a complier, and it is the mean difference of the true treatment (\citealp{imbens/angrist:1994} and \citealp{angrist/imbens/rubin:1996}).}
Since the measured treatment is not the true treatment, however, the size of compliers is not identified and therefore the local average treatment effect is not identified. 
The under-identification for the local average treatment effect is a consequence of the under-identification for the size of compliers; if I assumed no measurement error, I could compute the size of compliers based on the measured treatment and therefore the local average treatment effect would be the Wald estimand.\footnote{The Wald estimand in this paper is defined as the ratio of the intent-to-treat effect over the mean difference of the measured treatment between the two groups defined by the instrument.
Note that the Wald estimand is identified because it uses the measured treatment, but it is not the local average treatment effect because it does not use the true treatment.}

I take a worst case scenario approach with respect to the measurement error and allow for a general form of measurement error.
The only assumption concerning the measurement error is its independence of the instrumental variable. 
(Section \ref{sec3Less} dispenses with this assumption and shows that it is still possible to bound the local average treatment effect.)
I consider the following types of measurement error. 
First, the measurement error is nonclassical; that is, it can be dependent on the true treatment. 
The measurement error for a binary variable is always nonclassical. 
It is because the measurement error cannot be negative (positive) when the true variable takes the low (high) value. 
Second, I allow the measurement error to be endogenous (or differential); that is, the measured treatment can be dependent on the outcome conditional on the true treatment. 
For example, as \cite{black/sanders/taylor:2003} argue, the measurement error for educational attainment depends on the familiarity with the educational system in the U.S., and immigrants may have a higher rate of measurement error. 
At the same time, the familiarity with the U.S. educational system can be related to the English language skills, which can affect the labor market outcomes. 
\cite{bound/brown/mathiowetz:2001} also argue that measurement error is likely to be endogenous in some empirical applications. 
(In Appendix  D, I explore for the identifying power of the exogeneity assumption on the measurement error. The additional assumption yields a tighter sharp identified set, but I still cannot point identify the local average treatment effect in general.)
Third, there is no assumption concerning the marginal distribution of the measurement error.
It is not necessary to assume anything about the accuracy of the measurement.

In the presence of measurement error, I derive the identified set for the local average treatment effect (Theorem \ref{theorem1}). 
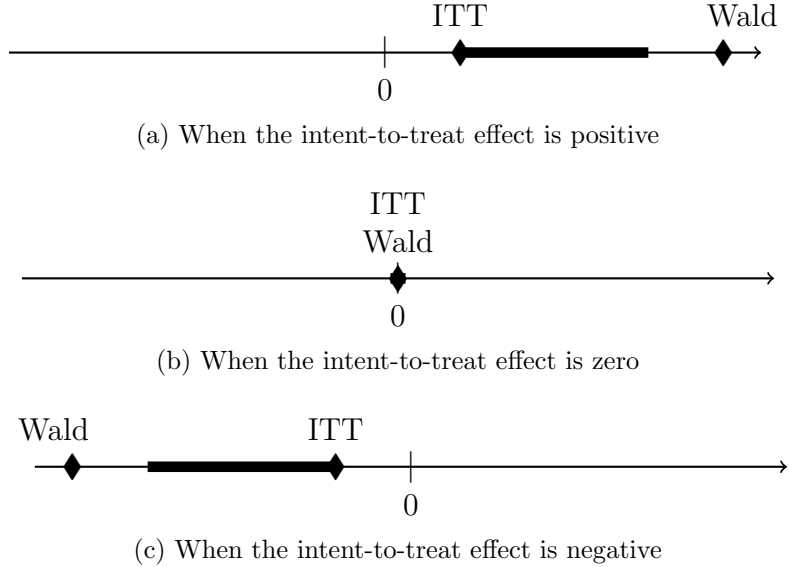
\begin{figure}
\centering
\subfloat[When the intent-to-treat effect is positive]{
\begin{tikzpicture}[thick]
  \draw[->] (0,1) to (10,1);
  \node[] at (5,1) {$|$};
  \node[] at (5,.5) {$0$};
  \node[] at (6,1) {$\blacklozenge$};
  \node[] at (6,1.5) {ITT};
  \draw [line width=4] (6,1) to (8.5,1);
  \node[] at (9.5,1) {$\blacklozenge$};
  \node[] at (9.75,1.5) {Wald};
\end{tikzpicture}
}
\\
\subfloat[When the intent-to-treat effect is zero]{
\begin{tikzpicture}[thick]
  \draw[->] (0,1) to (10,1);
  \node[] at (5,.5) {$0$};
  \node[] at (5,1) {$|$};
  \draw [line width=4] (4.9,1) to (5.1,1);
  \node[] at (5,1) {$\blacklozenge$};
  \node[] at (5,1.5) {Wald};
  \node[] at (5,2) {ITT};
\end{tikzpicture}
}
\\
\subfloat[When the intent-to-treat effect is negative]{
\begin{tikzpicture}[thick]
  \draw[->] (0,1) to (10,1);
  \node[] at (5,1) {$|$};
  \node[] at (5,.5) {$0$};
  \node[] at (4,1) {$\blacklozenge$};
  \node[] at (4,1.5) {ITT};
  \draw [line width=4] (1.5,1) to (4,1);
  \node[] at (.5,1) {$\blacklozenge$};
  \node[] at (.25,1.5) {Wald};
\end{tikzpicture}
}
\\
\caption{Identified set for the local average treatment effect.\label{GraphResult} ITT is the intent-to-treat effect and Wald is the Wald estimand. The thick line is the identified set for the local average treatment effect. Note that the identified set is $\{0\}$ when the intent-to-treat effect is zero.}
\end{figure}
Figure \ref{GraphResult} describes the relationship among the identified set for the local average treatment effect, the intent-to-treat effect, and the Wald estimand. 
First, the intent-to-treat effect has the same sign as the local average treatment effect. 
This is why Figure \ref{GraphResult} has three subfigures according to the sign of the intent-to-treat effect: (a) positive, (b) zero, and (c) negative. 
Second, the intent-to-treat effect is the sharp lower bound on the local average treatment effect in absolute value. 
Third, the Wald estimand is an upper bound on the local average treatment effect in absolute value. 
The Wald estimand is the probability limit of the instrumental variable estimator in my framework, which ignores the measurement error but controls only for the endogeneity. 
This point implies that an upper bound on the local average treatment effect is obtained by ignoring the measurement error. 
\cite{frazis/loewenstein:2003} obtain a similar result in the homogeneous treatment effect model. 
Last, but most importantly, the sharp upper bound in absolute value can be smaller than the Wald estimand. 
It is a potential cost of ignoring the measurement error and using the Wald estimand.
Even for analyzing only an upper bound on the local average treatment effect, it is recommended to take the measurement error into account, which can yield a smaller upper bound than the Wald estimand.
Section \ref{sec3.1} investigates when the Wald estimand coincide with the sharp upper bound. 

I extend the identification analysis to incorporate covariates other than the treatment variable. 
In this setting, the instrumental variable satisfies the exclusion restriction after conditioning covariates.
Based on the insights from \cite{abadie:2003} and \cite{frolich:2007}, I show that the identification strategy of this paper works in the presence of covariates. 

I construct a confidence interval for the local average treatment effect.
To construct the confidence interval, first, I approximate the identified set by discretizing the support of the outcome where the discretization becomes finer as the sample size increases.
The approximation for the identified set resembles many moment inequalities in \cite{menzel:2014} and \cite{chernozhukov/chetverikov/kato:2014}, who consider a finite but divergent number of moment inequalities. 
I apply a bootstrap method in \cite{chernozhukov/chetverikov/kato:2014} to construct a confidence interval with uniformly asymptotically valid asymptotic size control.
The confidence interval also rejects parameter values which do not belong to the sharp identified set. 
An empirical excise and a Monte Carlo simulation demonstrate a finite sample property of the proposed inference method. 
The empirical exercise is based on \cite{abadie:2003}, who studies the effects of 401(k) participation on financial savings, and considers a misclassification of the 401(k) participation.\footnote{The pension type is subject to a measurement error. See, for example, \cite{gustman/steinmeier/tabatabai:2007} for the pension type misclassification in the Health and Retirement Study.} 

As an extension, I consider the dependence between the instrument and the measurement error. 
In this case, there is no assumption on the measurement error, and therefore the measured treatment has no information on the local average treatment effect. 
Even without using the measured treatment, however, I can still apply the same identification strategy and obtain finite (but less tight) bounds on the local average treatment effect. 

Moreover, I offer a new use of repeated measurements as additional sources for identification. 
The existing practice of repeated measurements uses one of them as an instrumental variable, as in \cite{hausman/ichimura/newey/powell:1991}, \cite{hausman/newey/powell:1995}, \cite{mahajan:2006}, and \cite{hu:2008}.\footnote{It is worthwhile to mention that \cite{lewbel:2007} allows for a certain form of the endogeneity in a repeated measurement, under which a repeated measurement still satisfies some exclusion restriction.}
However, when the true treatment is endogenous, the repeated measurements are likely to be endogenous and are not good candidates for an instrumental variable. 
My identification strategy demonstrates that those variables are useful for bounding the local average treatment effect in the presence of measurement error, even if none of the repeated measurement are valid instrumental variables.

The remainder of this paper is organized as follows. 
Section \ref{sec1.1} explains several empirical examples motivating mismeasured endogenous treatments and Section \ref{sec1.2} reviews the related econometric literature.  
Section \ref{sec2} introduces mismeasured treatments in the framework of \cite{imbens/angrist:1994} and \cite{angrist/imbens/rubin:1996}.
Section \ref{sec3} constructs the identified set for the local average treatment effect.
I also discuss two extensions. One extension describes how repeated measurements tighten the identified set even if I cannot use any of the repeated measurements as an instrumental variable, and the other dispenses with independence between the instrument and the measurement error. 
Section \ref{sec4} proposes an inference procedure for the local average treatment effect. 
Section \ref{sec5.1} conducts  an empirical illustrations, and Section \ref{sec5} conducts Monte Carlo simulations. 
Section \ref{sec6} concludes. 
Appendix collects proofs and remarks.

\subsection{Examples for mismeasured endogenous treatments}\label{sec1.1}
I introduce several empirical examples in which binary treatments can be both endogenous and mismeasured at the same time. 
The first example is the return to schooling, in which the outcome is wages and the treatment is educational attainment, for example, whether a person has completed college or not.
Unobserved individual ability affects both the schooling decision and wage determination, which leads to the endogeneity of educational attainment in the wage equation (see, for example, \cite{card:2001}). 
Moreover, survey datasets record educational attainments based on the interviewee's answers, and these self-reported educational attainments are subject to measurement error. 
\cite{griliches:1977}, \cite{angrist/krueger:1999}, \cite{kane/rouse/staiger:1999}, \citealp{card:2001}, \cite{black/sanders/taylor:2003} have pointed out the mismeasurement of educational attainments. 
For example, \cite{black/sanders/taylor:2003} estimate that the 1990 Decennial Census has 17.7\% false positive rate of reporting a doctoral degree.

The second example is labor supply response to welfare program participation, in which the outcome is employment status and the treatment is welfare program participation. 
Self-reported welfare program participation in survey datasets can be mismeasured \citep{hernandez/pudney:2007}. 
The psychological cost for welfare program participation, welfare stigma, affects job search behavior and welfare program participation simultaneously; that is, welfare stigma may discourage individuals from participating in a welfare program, and, at the same time, affect an individual's effort in the labor market (see \cite{moffitt:1983} and \cite{besley/coate:1992} for a discussion on the welfare stigma).
Moreover, the welfare stigma gives welfare recipients some incentive not to reveal their participation status to the survey, which causes endogenous measurement error in that the unobserved individual heterogeneity affects both the measurement error and the outcome. 

The third example is the effect of a job training program on wages.
As it is similar to the return to schooling, unobserved individual ability plays a key role in this example. 
Self-reported completion of job training program is also subject to measurement error \citep{bollinger:1996}.  
\cite{frazis/loewenstein:2003} develop a methodology for evaluating a homogeneous treatment effect with mismeasured endogenous treatment, and apply their methodology to evaluate the effect of a job training program on wages.

The last example is the effect of maternal drug use on infant birth weight.
\cite{kaestner/joyce/wehbeh:1996} estimate that a mother tends to underreport her drug use, but, at the same time, she tends to report it correctly if she is a heavy user. 
When the degree of drug addiction is not observed, it becomes an individual unobserved heterogeneity which affects infant birth weight and the measurement in addition to the drug use. 

\subsection{Literature review}\label{sec1.2}
Here I summarize the related econometric literature. 
\cite{mahajan:2006}, \cite{lewbel:2007}, and \cite{hu:2008} use an instrumental variable to correct for measurement error in a binary (or discrete) treatment in the homogeneous treatment effect framework and they achieve nonparametric point identification of the average treatment effect.
They assume that the true treatment is exogenous, whereas I allow it to be endogenous.

Finite mixture models are related to my analysis.
I consider the unobserved binary treatment, whereas finite mixture models deal with unobserved type. 
\cite{henry/kitamura/salanie:2014} and \cite{henry/jochmans/salanie:2015} are the most closely related. 
They investigate the identification problem in finite mixture models, by using the exclusion restriction in which an instrumental variable only affects the mixing distribution of a type without affecting the component distribution (that is, the conditional distribution given the type). 
If I applied their approach directly to my framework, their exclusion restriction would imply conditional independence between the instrumental variable and the outcome given the true treatment. 
This conditional independence implies that the local average treatment effect does not exhibit essential heterogeneity \citep{heckman/schmierer/urzua:2010} and that the local average treatment effect is the mean difference between the control and treatment groups.\footnote{\label{FootNote3}This footnote uses the notation introduced in Section \ref{sec2}.
The conditional independence implies $E[Y\mid T^{\ast},Z]=E[Y\mid T^{\ast}]$.
Under this assumption, 
\begin{eqnarray*}
E[Y\mid Z]
&=&
P(T^{\ast}=1\mid Z)E[Y\mid Z,T^{\ast}=1]+P(T^{\ast}=0\mid Z)E[Y\mid Z,T^{\ast}=0]\\
&=&
P(T^{\ast}=1\mid Z)E[Y\mid T^{\ast}=1]+P(T^{\ast}=0\mid Z)E[Y\mid T^{\ast}=0]
\end{eqnarray*}
and therefore $\Delta E[Y\mid Z]=\Delta E[T^{\ast}\mid Z](E[Y\mid T^{\ast}=1]-E[Y\mid T^{\ast}=0])$. I obtain the equality 
$$
\frac{\Delta E[Y\mid Z]}{\Delta E[T^{\ast}\mid Z]}=E[Y\mid T^{\ast}=1]-E[Y\mid T^{\ast}=0]
$$
This above equation implies that the local average treatment effect does not depend on the compliers of consideration, which is in contrast with the essential heterogeneity of the treatment effect. 
Furthermore, since $E[Y\mid T^{\ast}=1]-E[Y\mid T^{\ast}=0]$ is the local average treatment effect, I do not need to care about the endogeneity.}  
Instead of applying the approaches in \cite{henry/kitamura/salanie:2014} and \cite{henry/jochmans/salanie:2015}, I use a different exclusion restriction in which the instrumental variable does not affect the outcome or the measured treatment directly. 

A few papers have applied an instrumental variable to a mismeasured binary regressor in the homogenous treatment effect framework. 
They include \cite{aigner:1973}, \cite{kane/rouse/staiger:1999}, \cite{bollinger:1996}, \cite{black/berger/scott:2000}, \cite{frazis/loewenstein:2003}, and \cite{ditragliaand/garica-jimeno:2015}. 
\cite{frazis/loewenstein:2003} and \cite{ditragliaand/garica-jimeno:2015} are the most closely related among them, since they allow for endogeneity. 
Here I allow for heterogeneous treatment effects, and I contribute to the heterogeneous treatment effect literature by investigating the consequences of the measurement errors in the treatment.

\cite{kreider/pepper:2007}, \cite{molinari:2008}, \cite{imai/yamamoto:2010}, and \cite{kreider/pepper/gundersen/joliffe:2012} apply a partial identification strategy for the average treatment effect to the mismeasured binary regressor problem by utilizing the knowledge of the marginal distribution for the true treatment.
Those papers use auxiliary datasets to obtain the marginal distribution for the true treatment.
\cite{kreider/pepper/gundersen/joliffe:2012} is the most closely related,  in that they allow for both treatment endogeneity and endogenous measurement error. 
My instrumental variable approach can be an an alternative strategy to deal with mismeasured endogenous treatment.
It is worthwhile because, as mentioned in \cite{schennach:2013}, the availability of an auxiliary dataset is limited in empirical research.  
Furthermore, it is not always the case that the results from auxiliary datasets is transported into the primary dataset \citep[][p.10]{carroll/crainiceanu/ruppert/stefanski:2012}, 

Some papers investigate mismeasured endogenous continuous variables, instead of binary variables.  
\cite{amemiya:1985,hsiao:1989,lewbel:1998,song/schennach/white:2015} consider nonlinear models with mismeasured continuous explanatory variables. 
The continuity of the treatment is crucial for their analysis, because they assume classical measurement error. 
The treatment in my analysis is binary and therefore the measurement error is nonclassical.
\cite{hu/shiu/woutersen:2015} consider mismeasured endogenous continuous variables in single index models. 
However, their approach depends on taking derivatives of the conditional expectations with respect to  the continuous variable. 
It is not clear if it can be extended to binary variables. 
\cite{song:2015} considers the semi-parametric model when endogenous continuous variables are subject to nonclassical measurement error. 
He assumes conditional independence between the instrumental variable and the outcome given the true treatment, which would impose some structure on the outcome equation when a treatment is binary (see Footnote \ref{FootNote3}). 
Instead I propose an identification strategy without assuming any structure on the outcome equation.

\cite{chalak:2013} investigates the consequences of measurement error in the instrumental variable instead of the treatment. 
He assumes that the treatment is perfectly observed, whereas I allow for it to be measured with error. 
Since I assume that the instrumental variable is perfectly observed, my analysis is not overlapped with \cite{chalak:2013}.

\cite{Manski:2003}, \cite{blundell/gosling/ichimura/meghir:2007}, and \cite{kitagawa:2010b} have similar identification strategy in the context of sample selection models. 
These papers also use the exclusion restriction of the instrumental variable for their partial identification results. 
Particularly, \cite{kitagawa:2010b} derives the integrated envelope from the exclusion restriction, which is similar to the total variation distance in my analysis because both of them are characterized as a supremum over the set of the partitions. 
First and the most importantly, I consider mismeasurement of the treatment, whereas the sample selection model considers truncation of the outcome. 
It is not straightforward to apply their methodologies in sample selection models into mismeasured treatment problem.
Second, I offer an inference method with uniform size control, but \cite{kitagawa:2010b} derives only point-wise size control. 
Last, \cite{blundell/gosling/ichimura/meghir:2007} and \cite{kitagawa:2010b} use their result for specification test, but I cannot use it to carry out a specification test because the sharp identified set of my analysis is always non-empty. 

Finally, \cite{calvi/lewbel/tommasi:2017} and \cite{yanagi:2017} have recently discussed identification issues of the local average treatment effect in the presence of a measurement error in the treatment variable. 
They are built on results in the previous draft of this paper \citep{ura:2015} to derive novel and important results when there are additional variables in a dataset: multiple measurements of the true treatment variable \citep{calvi/lewbel/tommasi:2017} or multiple instrumental variables \citep{yanagi:2017}. 
In contrast, the results of this paper are valid without these additional variables and only requires the assumptions in \cite{imbens/angrist:1994} and \cite{angrist/imbens/rubin:1996}.

\section{Local average treatment effect framework with misclassification}\label{sec2}
My analysis considers a mismeasured treatment in the framework of \cite{imbens/angrist:1994} and \cite{angrist/imbens/rubin:1996}. 
The objective is to evaluate the causal effect of a binary treatment $T^\ast\in\{0,1\}$ on an outcome $Y$, where $T^\ast=0$ represents the control group and $T^\ast=1$ represents the treatment group. 
To deal with endogeneity of $T^\ast$, I use a binary instrumental variable $Z\in\{0,1\}$ which shifts $T^\ast$ exogenously without any direct effect on $Y$.
The treatment $T^\ast$ of interest is not directly observed, and instead there is a binary measurement $T\in\{0,1\}$ for $T^\ast$. 
I put the $\ast$ symbol on $T^\ast$ to emphasize that the true treatment $T^\ast$ is unobserved. 
I allow $Y$ to be discrete, continuous or mixed; $Y$ is only required to have some known dominating finite measure $\mu_Y$ on the real line.
For example, $\mu_Y$ can be the Lebesgue measure or the counting measure. 
Let $\mathbf{Y}$ be the support for the random variable $Y$ and  $\mathbf{T}=\{0,1\}$ be the support for $T$. 

To describe the data generating process, I consider the counterfactual variables. 
$T_{z}^{\ast}$ is the counterfactual true treatment when $Z=z$.
$Y_{t^\ast}$ is the counterfactual outcome when $T^\ast=t^\ast$. 
$T_{t^\ast}$ is the counterfactual measured treatment when $T^\ast=t^\ast$. 
The individual treatment effect is $Y_1-Y_0$. 
It is not directly observed; $Y_0$ and $Y_1$ are not observed at the same time.
Only $Y_{T^\ast}$ is observable. 
Using the notation, the observed variables $(Y,T,Z)$ are generated by the three equations:  
\begin{eqnarray}
T&=&T_{T^\ast}\label{measurement}\\
Y&=&Y_{T^\ast}\label{outcome}\\
T^\ast&=&T^\ast_Z\label{treatment_assignment}.
\end{eqnarray}
Figure \ref{arrows} graphically describes the relationship among the instrument $Z$, the (unobserved) true treatment $T^{\ast}$, the measured treatment $T$, and the outcome $Y$. 
\begin{figure}
\centering
\begin{tikzpicture}[]
  \node[] at (0,0) {outcome $Y$};
  \node[] at (5,0) {true treatment $T^\ast$};
  \node[] at (5,-3) {measured treatment $T$}; 
  \node[] at (10,0) {instrument $Z$};
  \draw[->] (8,0) to (7,0);
  \draw[->] (3,0) to (2,0);
  \draw[->] (5,-1) to (5,-2);  
\end{tikzpicture}
\caption{Graphical representation of dependencies among variables\label{arrows}}
\end{figure}
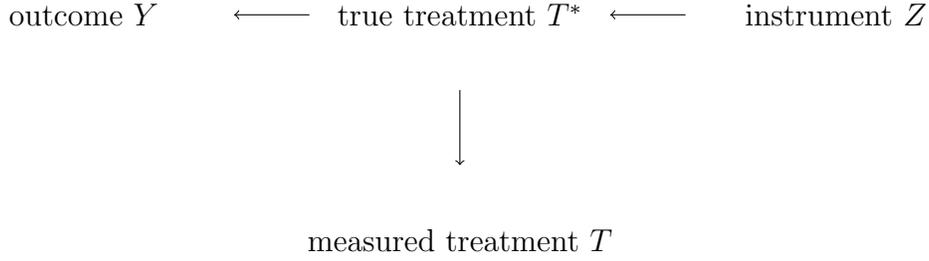  
(\ref{measurement}) is the measurement equation, which is the arrow from $T^\ast$ to $T$ in Figure \ref{arrows}. 
$T-T^\ast$ is the measurement error; $T-T^\ast=1$ (or $T_0=1$) represents a false positive and $T-T^\ast=-1$ (or $T_1=0$) represents a false negative.  
Equations (\ref{outcome}) and (\ref{treatment_assignment}) are the same as \cite{imbens/angrist:1994} and \cite{angrist/imbens/rubin:1996}. 
(\ref{outcome}) is the outcome equation, which is the arrow from $T^\ast$ to $Y$ in Figure \ref{arrows}. 
(\ref{treatment_assignment}) is the treatment assignment equation, which is the arrow from $Z$ to $T^\ast$ in Figure \ref{arrows}. 
A  potentially non-zero correlation between $(Y_0,Y_1)$ and $(T^\ast_{0},T^\ast_{1})$ causes an endogeneity problem.  

In a return to schooling analysis, $Y$ is wages, $T^\ast$ is the true indicator for college completion, $Z$ is the proximity to college, and $T$ is the self-reported college completion.
The treatment effect $Y_1-Y_0$ in the return to schooling is the effect of college completion $T^\ast$ on wages $Y$. 
The college completion is not correctly measured in a survey dataset, such that only the self report $T$ is observed. 

This section and Section \ref{sec3} impose only the following assumption. 
\begin{assumption}\label{assumption1}
(i) For each $t^\ast=0,1$, $Z$ is independent of $(T_{t^\ast},Y_{t^\ast},T^\ast_{0},T^\ast_{1})$.
(ii) $T^\ast_{1}\geq T^\ast_{0}$ almost surely.
(iii) $0<P(Z=1)<1$. 
\end{assumption} 
Assumption \ref{assumption1} (i) is the exclusion restriction and I consider stochastic independence instead of mean independence.
Although it is stronger than the minimal conditions for the identification for the local average treatment effect without measurement error, a large part of the existing applied papers assume stochastic independence \citep[][p.405]{huber/mellace:2014}. 
$Z$ is also independent of $T_{t^\ast}$ conditional on $(Y_{t^\ast},T^\ast_{0},T^\ast_{1})$, which is the only assumption on the measurement error for the identified set in Section \ref{sec3}. 
(Section \ref{sec3Less} even dispenses with this assumption.)
Assumption \ref{assumption1} (ii) is the monotonicity condition for the instrument, in which the instrument $Z$ increases the value of $T^\ast$ for all the individuals. 
\cite{dechaisemartin:2014} relaxes the monotonicity condition, and it can be shown in Appendix E that the identification results in my analysis still holds with a slight modification under the complier-defiers-for-marginals condition in \cite{dechaisemartin:2014}. 
Note that Assumption \ref{assumption1} does not include a relevance condition for the instrumental variable. 
The standard relevance condition $T^\ast_{1}\ne T^\ast_{0}$ does not affect the identification results in my analysis.
I will discuss the relevance condition in my framework after Theorem \ref{theorem1}. 
Assumption \ref{assumption1} (iii) excludes that $Z$ is constant. 

As I emphasized in the introduction, the framework here does not assume anything on measurement error $T_{t^\ast}$ except for its independence from $Z$. 
Assumption \ref{assumption1} does not impose any restriction on the marginal distribution of the measurement error $T_{t^\ast}$ or on the relationship between the measurement error $T_{t^\ast}$ and $(Y_{t^\ast},T^\ast_{0},T^\ast_{1})$. 
Particularly, the measurement error can be endogenous, that is, $T_{t^\ast}$ and $(Y_{t^\ast},T^\ast_{0},T^\ast_{1})$ can be correlated.\footnote{Although it has not been supported in validation data studies (e.g., \citealp{black/sanders/taylor:2003}), a majority of the literature on measurement error has assume that the measurement error is exogenous (\citealp{bound/brown/mathiowetz:2001}). I also explore for the identifying power of the exogenous measurement error assumption in Appendix D.}

I focus on the local average treatment effect, which is defined by 
$$
\theta=E[Y_1-Y_0\mid T_0^\ast<T_1^\ast].
$$ 
The local average treatment effect is the average of the treatment effect $Y_1-Y_0$ over the subpopulation (the compliers) whose treatment status is strictly affected by the instrument. 
\citet[][Theorem 1]{imbens/angrist:1994} show that the local average treatment effect equals 
$$
\frac{\Delta E[Y\mid Z]}{\Delta E[T^*\mid Z]},
$$
where I define $\Delta E[X\mid Z]= E[X\mid Z=1]-E[X\mid Z=0]$ for a random variable $X$. 
Note that $\Delta E[Y\mid Z]$ is the intent-to-treat effect, that is, the regression of $Y$ on $Z$.
The treatment is measured with error, and therefore the above fraction $\Delta E[Y\mid Z]/\Delta E[T^*\mid Z]$ is not the Wald estimand  
$$
\frac{\Delta E[Y\mid Z]}{\Delta E[T\mid Z]}.
$$
Since $\Delta E[T^*\mid Z]$ is not identified, I cannot identify the local average treatment effect.  
The failure of point identification comes purely from the measurement error, because the local average treatment effect would be point-identified under $T=T^\ast$.
In fact, my proposed methodology in this paper is essentially a bounding strategy of $\Delta E[T^*\mid Z]$ and I use the bound to construct the sharp identified set for the local average treatment effect. 

\section{Identified set for the local average treatment effect}\label{sec3}
This section show how the instrumental variable partially identifies the local average treatment effect in the framework of Section \ref{sec2}. 
Before defining the identified set, I express the local average treatment effect as a function of the underlying distribution $P^{\ast}$ of $(Y_0,Y_1,T_0,T_1,T^\ast_{0},T^\ast_{1},Z)$.
I use the $\ast$ symbol on $P^{\ast}$ to clarify that $P^{\ast}$ is the distribution of the unobserved variables. 
I denote the expectation operator $E$ by $E_{P^{\ast}}$ when I need to clarify the underlying distribution.  
The local average treatment effect is a function of the unobserved distribution $P^{\ast}$: 
$$
\theta(P^{\ast})= E_{P^{\ast}}[Y_1-Y_0\mid T_0^\ast<T_1^\ast].
$$
I denote by $\Theta$ the parameter space for the local average treatment effect $\theta$, that is, the set of $\int yf_1(y)d\mu_Y(y)-\int yf_0(y)d\mu_Y(y)$ where $f_0$ and $f_1$ are density functions dominated by the known probability measure $\mu_Y$. 
For example, $\Theta=[-1,1]$ when $Y$ is binary. 

The identified set is the set of parameter values for the local average treatment effect which is consistent with the distribution of the observed variables.  
I use $P$ for the distribution of the observed variables $(Y,T,Z)$ 
The equations (\ref{measurement}), (\ref{outcome}), and (\ref{treatment_assignment}) induce the distribution of the observables $(Y,T,Z)$ from the unobserved distribution $P^{\ast}$, and I denote by $P^{\ast}_{(Y,T,Z)}$ the induced distribution.
When the distribution of $(Y,T,Z)$ is $P$, the set of $P^{\ast}$ which induces $P$ is $\{P^{\ast}\in\mathcal{P}^{\ast}: P=P^{\ast}_{(Y,T,Z)}\}$, where $\mathcal{P}^{\ast}$ is the set of $P^{\ast}$'s satisfying Assumptions \ref{assumption1}. 
For every distribution $P$ of $(Y,T,Z)$, the (sharp) identified set for the local average treatment effect is defined as 
$\Theta_I(P)=\{\theta(P^{\ast})\in\Theta: P^{\ast}\in\mathcal{P}^{\ast}\mbox{ and }P=P^{\ast}_{(Y,T,Z)}\}$.

\citet[][Theorem 1]{imbens/angrist:1994} provides a relationship between $\Delta E[Y\mid Z]$ and the local average treatment effect: 
\begin{equation}\label{eqIA}
\theta(P^{\ast})P^{\ast}(T_0^\ast<T_1^\ast)=\Delta E_{P^{\ast}}[Y\mid Z],
\end{equation}
This equation gives the two pieces of information of $\theta(P^{\ast})$. 
First, the sign of $\theta(P^{\ast})$ is the same as $\Delta E_{P^{\ast}}[Y\mid Z]$.
Second, the absolute value of $\theta(P^{\ast})$ is at least the absolute value of $\Delta E_{P^{\ast}}[Y\mid Z]$.
The following lemma summarizes these two pieces. 
\begin{lemma}\label{lemma1}
Under Assumption \ref{assumption1},
$$
\theta(P^{\ast})\Delta E_{P^{\ast}}[Y\mid Z]\geq 0
$$
$$
|\theta(P^{\ast})|\geq |\Delta E_{P^{\ast}}[Y\mid Z]|.
$$
\end{lemma}

I derive a new implication from the exclusion restriction for the instrumental variable in order to obtain an upper bound on $\theta(P^{\ast})$ in absolute value. 
To explain the new implication, I introduce the total variation distance, which is the $L^1$ distance between the distribution $f_1$ and $f_0$:
For any random variable $X$, define 
$$
TV_X=\frac{1}{2}\int |f_{X\mid Z=1}(x)-f_{X\mid Z=0}(x)|d\mu_X(x),
$$
where $\mu_X$ is a dominating measure for the distribution of $X$.
\begin{lemma}\label{lemma2}
Under Assumption \ref{assumption1},
$$
TV_{(Y,T)}\leq TV_{T^\ast}=P^{\ast}(T_0^\ast<T_1^\ast).
$$
\end{lemma}

The first term, $TV_{(Y,T)}$, in Lemma \ref{lemma2} reflects the dependency of $f_{(Y,T)\mid Z=z}(y,t)$ on $z$, and it can be interpreted as the magnitude of the distributional effect of $Z$ on $(Y,T)$. 
The second and third terms, $TV_{T^\ast}$ and $P^{\ast}(T_0^\ast<T_1^\ast)$, are the effect of the instrument $Z$ on the true treatment $T^\ast$. 
Based on Lemma \ref{lemma2}, the magnitude of the effect  of $Z$ on $T^\ast$ is no smaller than the magnitude of  the effect  of $Z$ on $(Y,T)$. 

The new implication in Lemma \ref{lemma2} gives a lower bound on $P^{\ast}(T_0^\ast<T_1^\ast)$ and therefore yields an upper bound on the local average treatment effect in absolute value, combined with equation (\ref{eqIA}).
Therefore, I use these relationships to derive an upper bound on the local average treatment effect in absolute value, that is, 
$$
|\theta(P^{\ast})|=\frac{|\Delta E_{P^{\ast}}[Y\mid Z]|}{P^{\ast}(T_0^\ast<T_1^\ast)}\leq \frac{|\Delta E_{P^{\ast}}[Y\mid Z]|}{TV_{(Y,T)}}
$$
as long as $TV_{(Y,T)}>0$. 

Theorem \ref{theorem1} shows that the above observations characterize the sharp identified set for the local average treatment effect.   
\begin{theorem}\label{theorem1}
Suppose that Assumption \ref{assumption1} holds, and consider an arbitrary data distribution $P$ of $(Y,T,Z)$. 
The identified set $\Theta_I(P)$ for the local average treatment effect is characterized as follows: 
$\Theta_I(P)=\Theta$ if $TV_{(Y,T)}=0$; otherwise, 
$$
\Theta_I(P)
=
\begin{cases}
\left[\Delta E_P[Y\mid Z],\frac{\Delta E_P[Y\mid Z]}{TV_{(Y,T)}}\right]&\mbox{ if }\Delta E_P[Y\mid Z]>0\\
\{0\}&\mbox{ if }\Delta E_P[Y\mid Z]=0\\
\left[\frac{\Delta E_P[Y\mid Z]}{TV_{(Y,T)}},\Delta E_P[Y\mid Z]\right]&\mbox{ if }\Delta E_P[Y\mid Z]<0.
\end{cases}
$$
\end{theorem}

The total variation distance $TV_{(Y,T)}$ plays two roles in determining the sharp identified set in this theorem. 
First, $TV_{(Y,T)}$ measures the strength of the instrumental variable, that is, $TV_{(Y,T)}>0$ is the relevance condition in my identification analysis. 
When $TV_{(Y,T)}>0$, the interval in the above theorem is always nonempty and bounded, which implies that $Z$ has some identifying power for the local average treatment effect. 
By contrast, $TV_{(Y,T)}=0$ means that the instrumental variable $Z$ does not affect $Y$ and $T$, in which case $Z$ has no identifying power for the local average treatment effect.  
In this case, $f_{(Y,T)\mid Z=1}=f_{(Y,T)\mid Z=0}$ almost everywhere over $(y,t)$ and particularly $\Delta E_P[Y\mid Z]=0$.
Note that all the three inequalities in Theorem \ref{theorem1} have no restriction on $\theta$ in this case. 
Second, $TV_{(Y,T)}$ determines the length of the sharp identified set. 
The length is $|\Delta E_P[Y\mid Z]|(TV_{(Y,T)}^{-1}-1)$, which is a decreasing function in $TV_{(Y,T)}$. 

In general, the lower and upper bounds of the sharp identified set are not equal to the local average treatment effect. 
The lower bound is weakly smaller (in the absolute value) than the local average treatment effect, because the size of the compliers is weakly smaller than one. 
The upper bound is weakly larger (in the absolute value) than the local average treatment effect, because $TV_{(Y,T)}$ is weakly smaller than the size of the compliers due to the mis-measurement of the treatment variable.

The standard relevance condition $\Delta E_P[T\mid Z]\ne 0$ is not required in Theorem \ref{theorem1}.
$\Delta E_P[T\mid Z]\ne 0$ is a necessary condition to define the Wald estimand, but the sharp identified set does not  depend directly on the Wald estimand. 
In fact, $TV_{(Y,T)}>0$ in Theorem \ref{theorem1} is weaker than $\Delta E_P[T\mid Z]\ne 0$. 

Note that the sharp identified set is always non-empty. 
There is no testable implications on the distribution of the observed variables, and therefore it is impossible to conduct a specification test for Assumption \ref{assumption1}.  

\subsection{Wald estimand and the identified set}\label{sec3.1}
The Wald estimand $\Delta E_P[Y\mid Z]/\Delta E_P[T\mid Z]$ can be outside the identified set.
One necessary and sufficient condition for the Wald estimand to be included in the identified set is given as follows. 
\begin{lemma}\label{Wald_lemma}
The Wald estimand is in the identified set if and only if 
\begin{equation}
f_{(Y,T)\mid Z=1}(y,1)\geq f_{(Y,T)\mid Z=0}(y,1)\mbox{ and }f_{(Y,T)\mid Z=1}(y,0)\leq f_{(Y,T)\mid Z=0}(y,0).
\label{TESTABLEIMPLICAT}
\end{equation}
\end{lemma}

This condition in (\ref{TESTABLEIMPLICAT}) are the testable implications from the the local average treatment effect framework without measurement error (\citealp{balke/pearl:1997} and \citealp{heckman/vytlacil:2005}). 
The recent papers by \cite{huber/mellace:2014}, \cite{kitagawa:2014}, and \cite{mourifie/wan:2014} propose the testing procedures for  (\ref{TESTABLEIMPLICAT}). 
Based on the results in Theorem \ref{theorem1}, their testing procedures are re-interpreted as a test for the null hypothesis that the Wald estimand is inside the sharp upper bound on the local average treatment effect.\footnote{Unfortunately, (\ref{TESTABLEIMPLICAT}) cannot be used for testing the existence of a measurement error. Even if there is non-zero measurement error, (\ref{TESTABLEIMPLICAT}) can still hold.}

\subsection{Conditional exogeneity of the instrumental variable}
As in \cite{abadie:2003} and \cite{frolich:2007}, this section considers the conditional exogeneity of the instrumental variable $Z$ in which $Z$ is exogenous given a set of covariates $V$, which weaker than the unconditional exogeneity in Assumption \ref{assumption1}.  
\begin{assumption}\label{assumption1condi}
There is some variable $V$ taking values in a set $\mathbf{V}$ satisfying the following properties. 
(i) For each $t^\ast=0,1$, $Z$ is conditionally independent of $(T_{t^\ast},Y_{t^\ast},T^\ast_{0},T^\ast_{1})$ given $V$.
(ii) $T^\ast_{1}\geq T^\ast_{0}$ almost surely.
(iii) $0<P(Z=1\mid V)<1$. 
\end{assumption} 

I define the $V$-conditional total variation distance by 
$$
TV_{X\mid V}=\frac{1}{2}\int |f_{X\mid Z=1,V}(x)-f_{X\mid Z=0,V}(x)|d\mu_X(x).
$$
Note that $TV_{X\mid V}$ is a random variable as a function of $V$. 
Under the conditional exogeneity of $Z$, Theorem \ref{theorem1} becomes as follows. 
\begin{theorem}\label{theorem1conditional}
Suppose that Assumption \ref{assumption1condi} holds, and consider an arbitrary data distribution $P$ of $(Y,T,Z,V)$. 
The identified set $\Theta_I(P)$ for the local average treatment effect is characterized as follows: $\Theta_I(P)=\Theta$ if $E_P[TV_{(Y,T)\mid V}]=0$; otherwise, 
$$
\Theta_I(P)
=
\begin{cases}
\left[E_P[\Delta E_P[Y\mid Z,V]],\frac{E_P[\Delta E_P[Y\mid Z,V]]}{E_P[TV_{(Y,T)\mid V}]}\right]&\mbox{ if }E_P[\Delta E_P[Y\mid Z,V]]>0\\
\{0\}&\mbox{ if }E_P[\Delta E_P[Y\mid Z,V]]=0\\
\left[\frac{E_P[\Delta E_P[Y\mid Z,V]]}{TV_{(Y,T)}},E_P[\Delta E_P[Y\mid Z,V]]\right]&\mbox{ if }E_P[\Delta E_P[Y\mid Z,V]]<0.
\end{cases}
$$
\end{theorem}

\subsection{Identifying power of repeated measurements}
The identification strategy in the above analysis offers a new use of repeated measurements as additional sources for identification. 
Repeated measurements \citep[for example,][]{hausman/ichimura/newey/powell:1991} is a popular approach in the literature on measurement error, but they cannot be instrumental variables in this framework. 
This is because the true treatment $T^\ast$ is endogenous and it is natural to suspect that a measurement of $T^\ast$ is also endogenous.
The more accurate the measurement is, the more likely it is to be endogenous.  
Nevertheless, the identification strategy incorporates repeated measurements as an additional information to tighten the identified set for the local average treatment effect, when they are coupled with the instrumental variable $Z$.
Unlike the other paper on repeated measurements, I do not need to assume the independence of measurement errors among multiple measurements.
The strategy also benefits from having more than two measurements unlike \cite{hausman/ichimura/newey/powell:1991} who achieve point identification with two measurements.

Consider a repeated measurement $R$ for $T^\ast$.
I do not require that $R$ is binary, so $R$ can be discrete or continuous. 
Like $T=T_{T^\ast}$, I model $R$ using the counterfactual outcome notations. 
$R_1$ is a counterfactual second measurement when the true treatment $T^\ast$ is $1$, and $R_0$ is a counterfactual second measurement when the true treatment $T^\ast$ is $0$. 
Then the data generation of $R$ is 
$$
R=R_{T^\ast}.
$$ 
I strengthen  Assumption \ref{assumption1} by assuming that the instrumental variable $Z$ is independent of $R_{t^\ast}$ conditional on $(Y_{t^\ast},T_{t^\ast},T^\ast_{0},T^\ast_{1})$. 
\begin{assumption}\label{assumption8}
(i) $Z$ is independent of $(R_{t^\ast},T_{t^\ast},Y_{t^\ast},T^\ast_{0},T^\ast_{1})$ for each $t^\ast=0,1$.
(ii) $T^\ast_{1}\geq T^\ast_{0}$ almost surely.
(iii) $0<P(Z=0)<1$. 
\end{assumption} 
Note that I do not assume the independence between $R_{t^\ast}$ and $T_{t^\ast}$, where the independence between the measurement errors is a key assumption when the repeated measurement is an instrumental variable. 
Assumption \ref{assumption8} tightens the identified set for the local average treatment effect as follows.

The requirement on $R$ does not restrict $R$ to have the same support as $T^\ast$. 
In fact, $R$ can be any variable which depends on $T^\ast$. 
For example, $R$ can be another outcome variable than $Y$. 

\begin{theorem}\label{theorem4}
Suppose that Assumption \ref{assumption8} holds, and consider an arbitrary data distribution $P$ of $(R,Y,T,Z)$. 
The identified set $\Theta_I(P)$ for the local average treatment effect is characterized as follows: 
$\Theta_I(P)=\Theta$ if $TV_{(R,Y,T)}=0$; otherwise, 
$$
\Theta_I(P)
=
\begin{cases}
\left[\Delta E_P[Y\mid Z],\frac{\Delta E_P[Y\mid Z]}{TV_{(R,Y,T)}}\right]&\mbox{ if }\Delta E_P[Y\mid Z]>0\\
\{0\}&\mbox{ if }\Delta E_P[Y\mid Z]=0\\
\left[\frac{\Delta E_P[Y\mid Z]}{TV_{(R,Y,T)}},\Delta E_P[Y\mid Z]\right]&\mbox{ if }\Delta E_P[Y\mid Z]<0.
\end{cases}
$$
\end{theorem}

The identified set in Theorem \ref{theorem4} is weakly smaller than the identified set in Theorem \ref{theorem1}. 
The total variation distance $TV_{(R,Y,T)}$ in Theorem \ref{theorem4} is weakly larger than that in Theorem \ref{theorem1}, because, using the triangle inequality, 
\begin{eqnarray*}
TV_{(R,Y,T)}
&=&
\frac{1}{2}\sum_{t=0,1}\iint|(f_{(R,Y,T)\mid Z=1}-f_{(R,Y,T)\mid Z=0})(r,y,t)|d\mu_R(r)d\mu_Y(y)\\
&\geq&
\frac{1}{2}\sum_{t=0,1}\int|\int(f_{(R,Y,T)\mid Z=1}-f_{(R,Y,T)\mid Z=0})(r,y,t)d\mu_R(r)|d\mu_Y(y)\\
&=&
\frac{1}{2}\sum_{t=0,1}\int|(f_{(Y,T)\mid Z=1}-f_{(Y,T)\mid Z=0})(y,t)|d\mu_Y(y)\\
&=&
TV_{(Y,T)}
\end{eqnarray*}
and the strict inequality holds unless the sign of $(f_{(R,Y,T)\mid Z=1}-f_{(R,Y,T)\mid Z=0})(r,y,t)$ is constant in $r$ for every $(y,t)$.
Therefore, it is possible to test whether the repeated measurement $R$ has additional information, by testing whether the sign of $(f_{(R,Y,T)\mid Z=1}-f_{(R,Y,T)\mid Z=0})(r,y,t)$ is constant in $r$.

\subsection{Dependence between measurement error and instrumental variable}\label{sec3Less}
It is still possible to apply the same identification strategy and obtain finite (but less tight) bounds on the local average treatment effect, even without the independence between the instrumental variable and the measurement error. (Assumption \ref{assumption1} (i) implies that $Z$ is independent of $T_{t^\ast}$ for each $t^\ast=0,1$.) 
Instead Assumption \ref{assumption1} is weakened to allow for the measurement error $T_{t^\ast}$ to be correlated with the instrumental variable $Z$.  
\begin{assumption}\label{assumption1-less}
(i) $Z$ is independent of $(Y_{t^\ast},T^\ast_{0},T^\ast_{1})$ for each $t^\ast=0,1$.
(ii) $T^\ast_{1}\geq T^\ast_{0}$ almost surely.
(iii) $0<P(Z=0)<1$. 
\end{assumption} 

Theorem \ref{theorem1-less} shows that the above observations characterize the identified set for the local average treatment effect under Assumption \ref{assumption1-less}.   
\begin{theorem}\label{theorem1-less}
Suppose that Assumption \ref{assumption1-less} holds, and consider an arbitrary data distribution $P$ of $(Y,T,Z)$. 
The identified set $\Theta_I(P)$ for the local average treatment effect is characterized as follows: 
$\Theta_I(P)=\Theta$ if $TV_Y=0$; otherwise, 
$$
\Theta_I(P)
=
\begin{cases}
\left[\Delta E_P[Y\mid Z],\frac{\Delta E_P[Y\mid Z]}{TV_Y}\right]&\mbox{ if }\Delta E_P[Y\mid Z]>0\\
\{0\}&\mbox{ if }\Delta E_P[Y\mid Z]=0\\
\left[\frac{\Delta E_P[Y\mid Z]}{TV_Y},\Delta E_P[Y\mid Z]\right]&\mbox{ if }\Delta E_P[Y\mid Z]<0.
\end{cases}
$$
\end{theorem}

The difference from Theorem \ref{theorem1} is that Theorem \ref{theorem1-less} does not depend on the measured treatment $T$. 
Although it is observed in the dataset,  $T$ does not have any information on the local average treatment effect because Assumption \ref{assumption1-less} does not restrict $T$. 
When $TV_Y>0$, there are nontrivial upper and lower bounds on the local average treatment effect even without using the measured treatment $T$.

\section{Inference}\label{sec4}
Based on the sharp identified set in the presence of covariates (Theorem \ref{theorem1conditional}), this section constructs a confidence interval for the local average treatment effect based on an i.i.d. sample $\{W_i:1\leq i\leq n\}$ of $W=(Y,T,Z,V)$. 
The confidence interval described below controls the asymptotic size uniformly over a class of data generating processes, and rejects all the fixed alternatives. 

The identified set in \ref{theorem1conditional} is characterized by moment inequalities as follows. 
\begin{lemma}\label{lemma3}
Let $P$ be an arbitrary data distribution of $W=(Y,T,Z,V)$. 
Under Assumption \ref{assumption1condi}, 
$\Theta_I(P)$ is the set of $\theta\in\Theta$ in which
\begin{eqnarray}
&&E_P\left[-\frac{Z-\pi(V)}{\pi(V)(1-\pi(V))}\mathrm{sgn}(\theta)Y\right]\leq 0\label{ID_Cond1}\\
&&E_P\left[\frac{Z-\pi(V)}{\pi(V)(1-\pi(V))}\mathrm{sgn}(\theta)Y-|\theta|\right]\leq 0\label{ID_Cond2}\\
&&E_P\left[\frac{Z-\pi(V)}{\pi(V)(1-\pi(V))}\left(|\theta|h(Y,T,V)-\mathrm{sgn}(\theta)Y\right)\right]\leq 0\mbox{ for all }h\in\mathbf{H}\label{ID_Cond3},
\end{eqnarray}
where $\pi(V)=P(Z=1\mid V)$, $\mathbf{H}$ is the set of measurable functions on $\mathbf{Y}\times\mathbf{T}\times\mathbf{V}$ taking a value in $\{-.5,.5\}$ and $\mathrm{sgn}(x)\equiv\mathbbm{1}\{x\geq 0\}-\mathbbm{1}\{x<0\}$. 
\end{lemma}

I construct a $(1-\alpha)$-confidence interval for the local average treatment effect $\theta$ with treating $\pi$ as a nuisance parameter for given $\alpha\in(0,0.5)$. 
I assume that a $(1-\delta)$-confidence interval $\mathcal{C}_{\pi,n}(\delta)$ for $\pi$ is available for researchers for given $\delta\in(0,\alpha)$.
Given $\mathcal{C}_{\pi,n}(\delta)$, I construct the $(1-\alpha-\delta)$-confidence interval $\mathcal{C}_{\theta,n}(\alpha+\delta)$ for the local average treatment effect as
$$
\mathcal{C}_{\theta,n}(\alpha+\delta)=\bigcup_{\pi\in\mathcal{C}_{\pi,n}(\delta)}\{\theta\in\Theta: T(\theta,\pi)\leq c(\alpha,\theta,\pi)\},
$$
where $T(\theta,\pi)$ and $c(\alpha,\theta,\pi)$ are defined below using the bootstrap-based testing \citep{chernozhukov/chetverikov/kato:2014}. 

The number of the moment inequalities in Lemma \ref{lemma3} can be finite or infinite, which determines whether some of the existing methods can be applied directly to the inference on the local average treatment effect. 
When $(Y,V)$ has finite supports and therefore $\mathbf{H}$ is finite, the sharp identified set is characterized by a finite number of inequalities, and therefore I can apply inference methods based on unconditional moment inequalities.\footnote{The literature on conditional and unconditional moment inequality models is broad and growing. See \cite{canay/shaikh:2016} for a recent survey on this literature.}
To the best of my knowledge, however, inference for the local average treatment effect in my framework does not fall directly into the existing moment inequality models when either $Y$ or $V$ is continuous. When either $Y$ or $V$ is continuous, the sharp identified set is characterized by an uncountably infinite number of inequalities. In the current literature on the partially identified parameters, an infinite number of moment inequalities are mainly considered in the context of conditional moment inequalities. The identified set in this paper is not characterized by conditional moment inequalities.\footnote{\cite{chernozhukov/lee/rosen:2013} also considers an infinite number of unconditional moment inequalities in which the moment functions are continuously indexed by a compact subset in a finite dimensional space. It is not straightforward to verify the continuity condition (Condition C.1 in their paper) for the moment inequalities in Lemma \ref{lemma3}, in which the moment functions need to be continuously indexed by a compact subset of the finite dimensional space.}\footnote{\cite{andrews/shi:2016} considers an infinite number of unconditional moment inequalities in which the moment functions satisfies manageability condition. I cannot apply their approach here because $h(Y,T,V)$ takes discrete values in $\{-.5,.5\}$ and then the packing numbers depends on the sample size.} 

I considers a sequence of finite sets $\mathbf{H}_n$ which converges to $\mathbf{H}$ as a sample size increases. (The convergence is formally defined in Assumption \ref{H_conditions}, and an example for $\mathbf{H}_n$ appears after Assumption \ref{H_conditions}.)   
Note that, when $\mathbf{H}$ is finite, $\mathbf{H}_n$ can be equal to $\mathbf{H}$.
If $\mathbf{H}$ is replaced with $\mathbf{H}_n$ in Lemma \ref{lemma3}, the number of the moment inequalities becomes finite. 
At the same time, as $\mathbf{H}_n$ approaches to $\mathbf{H}$, the approximation error from using $\mathbf{H}_n$ converges to zero, and the number of the inequalities can be increasing, particularly diverging to the infinity when $\mathbf{H}$ includes infinite elements.
The approximated identified set is characterized by a finite number of the following moment inequalities: 
\begin{eqnarray}
&&E_P\left[-\frac{Z-\pi(V)}{\pi(V)(1-\pi(V))}\mathrm{sgn}(\theta)Y\right]\leq 0\label{Approx1}\\
&&E_P\left[\frac{Z-\pi(V)}{\pi(V)(1-\pi(V))}\mathrm{sgn}(\theta)Y-|\theta|\right]\leq 0\label{Approx2}\\
&&E_P\left[\frac{Z-\pi(V)}{\pi(V)(1-\pi(V))}\left(|\theta|h(Y,T,V)-\mathrm{sgn}(\theta)Y\right)\right]\leq 0\mbox{ for all }h\in\mathbf{H}_n.\label{Approx3}
\end{eqnarray}
Denote by $p_n$ the resulting number of moment inequalities, that is, the number of elements in $\mathbf{H}_n$ plus $2$.  
Note that, when $K_n=1$, the moment inequalities in (\ref{Approx3}) is equivalent to using the Wald estimand as the upper bound for $|\theta|$.

For the size $\alpha\in(0,.5)$, I construct a test statistic $T(\theta,\pi)$ and a critical value $c(\alpha,\theta,\pi)$ via the multiplier bootstrap \citep{chernozhukov/chetverikov/kato:2014} for many moment inequality models (described in Section \ref{CCKbootstrap}).\footnote{In this paper I focus on the one-step multiplier bootstrap in \citep{chernozhukov/chetverikov/kato:2014}. It is also possible to use the two- or three-step empirical/multiplier bootstrap in this paper, but I do not compare them because the comparison of these methods is above the scope of this paper. }
\cite{chernozhukov/chetverikov/kato:2014} studies the testing problem for moment inequality models in which the number of the moment inequalities is finite but growing. 
Since the number of the moment inequalities in (\ref{Approx1})-(\ref{Approx3}) is finite but growing, their results are applicable to construct a confidence interval based on (\ref{Approx1})-(\ref{Approx3}). 

\begin{assumption}\label{asympt_assumption}
Given positive constants $C_2$ and $\eta$, the class of data generating processes, denoted by $\mathcal{P}_0$, and the parameter spaces $\Theta\times\Pi$ satisfy  
\begin{itemize}
\item[(i)] $\max\{E_P[Y^3]^{2/3},E_P[Y^4]^{1/2}\}<C_2$,
\item[(ii)] $\Theta\subset\mathbb{R}$ is bounded, 
\item[(iii)] The random variable inside $E_P$ in (\ref{Approx1})-(\ref{Approx3}) has a non-zero variance for every $j=1,\ldots,p_n$ and every $\theta\in\Theta$,
\item[(iv)] $\liminf_{n\rightarrow\infty}\inf_{P\in\mathcal{P}_0}P(\pi(P)\in \mathcal{C}_{\pi,n}(\delta))\geq 1-\delta$,
\item[(v)] $\eta<\pi(V)<1-\eta$ for every $\pi\in\Pi$.
\end{itemize}
\end{assumption}
The first assumption (i) is a regularity condition. 
The second assumption (ii) requires researchers to know ex ante upper and lower bounds on the parameter. 
The third assumption (iii) guarantees that the test statistic is well-defined.
The fourth assumption (iv) is that the confidence interval for $\pi$ controls the size uniformly over $\mathcal{P}_0$.
The last assumption (v) is that the propensity score $\pi(v)=P(Z=1\mid V=v)$ is bounded away from zero and one.

In this paper I assume that $\{\mathbf{H}_n\}$ satisfies the following conditions. 
\begin{assumption}\label{H_conditions}  
(i) $\mathbf{H}_n\subset\mathbf{H}_{n+1}$. (ii) The convergence  
\begin{equation}\label{approx_vanish}
\sup_{h\in\mathbf{H}}E_P\left[\frac{Z-\pi(V)}{\pi(V)(1-\pi(V))}h(Y,T,V)\right]-\max_{h\in\mathbf{H}_n}E_P\left[\frac{Z-\pi(V)}{\pi(V)(1-\pi(V))}h(Y,T,V)\right]\rightarrow 0
\end{equation}
holds uniformly over $\pi\in\Pi$ and $P\in\mathcal{P}_0$.
(iii) The number of elements in $\mathbf{H}_n$ satisfies 
\begin{equation}\label{number_moments}
\log^{7/2}(p_nn)\leq C_1n^{1/2-c_1}\mbox{ and }\log^{1/2}p_n\leq C_1n^{1/2-c_1}
\end{equation}
for some $c_1\in (0,1/2)$ and $C_1>0$.
\end{assumption}

An example of $\mathbf{H}_n$ is obtained by discretizing $Y$. 
Consider a partition $I_{n,1},\ldots,I_{n,K_n}$ over $\mathbf{Y}\times\mathbf{T}\times\mathbf{V}$, in which the intervals $I_{n,k}$ and the grid size $K_n$ depend on the sample size $n$.
Let $h_{n,j}$ be a generic function of $\mathbf{Y}\times\mathbf{T}\times\mathbf{V}$ into $\{-.5,.5\}$ that is constant over $I_{n,k}$ for every $1\leq k\leq K_n$.
Let $\mathbf{H}_n=\{h_{n,1},\ldots,h_{n,2^{K_n}}\}$ be the set of all such functions.
Lemma \ref{H_example} shows that this construction of $\mathbf{H}_n$ satisfies Eq. (\ref{approx_vanish}) under conditions on $I_{n,k}$ and $f_{(Y,T)\mid Z=z}$.
The conditions in Lemma \ref{H_example} guarantee that the approximation error from the discretization vanishes as the sample size $n$ increases.

It is worthwhile to mention that, when $K_n=1$, the implied upper bound in (\ref{Approx3}) is equal to the Wald estimand. 
It can be smaller than the Wald estimand as long as $K_n\geq 2$, 

\begin{lemma}\label{H_example}
Assumption \ref{H_conditions} holds if 
\begin{itemize}
\item[(i)] the partition $I_{n+1,1},\ldots,I_{n+1,K_{n+1}}$ is a refinement of the partition $I_{n,1},\ldots,I_{n,K_n}$;
\item[(ii)] $p_n=2^{K_n}+2$ satisfies (\ref{number_moments});
\item[(iii)] there is a positive constant $D_1$ such that $I_{n,k}$ is a subset of some open ball with radius $D_1/K_n$ in $\mathbf{Y}\times\mathbf{T}\times\mathbf{V}$;  and 
\item[(iv)] the density function $f_{(Y,T)\mid Z=z,V}$ is H\"older continuous in $(y,t,v)$ with the H\"older constant $D_0$ and exponent $d$.
\end{itemize}
\end{lemma}

Theorem \ref{fixed_alternative} shows asymptotic properties of the confidence interval $\mathcal{C}_{\theta,n}(\alpha+\delta)$. The first result (i) is the uniform asymptotic size control and the second result (ii) is the consistency against all the fixed alternatives.
\begin{theorem}\label{fixed_alternative}
Suppose that Assumptions \ref{asympt_assumption} and \ref{H_conditions} hold.
(i) The confidence interval controls the asymptotic size uniformly: 
$$
\liminf_{n\rightarrow\infty}\inf_{P\in\mathcal{P}_0, \theta\in\Theta_I(P)}P(\theta\in \mathcal{C}_{\theta,n}(\alpha+\delta))\geq 1-\alpha-\delta,
$$ 
(ii) If Eq. (\ref{approx_vanish}) holds, the confidence interval excludes all the fixed alternatives: 
$$
\lim_{n\rightarrow\infty}P(\theta\in \mathcal{C}_{\theta,n}(\alpha+\delta))=0\mbox{ for every }(\theta,P)\in\Theta\times\mathcal{P}_0\mbox{ with }\theta\not\in\Theta_I(P).
$$
\end{theorem}

\section{Empirical illustrations}\label{sec5.1}
This section studies the effects of 401(k) participation on financial savings using the inference method in Section \ref{sec4}. 
I introduce a measurement error problem to the analysis of \cite{abadie:2003}, which investigates the local average treatment effect using the eligibility for 401(k) program. 
The robustness to misclassification is empirically relevant, because the retirement pension plan type is subject to a measurement error in survey datasets.
Using the Health and Retirement Study, for example, \cite{gustman/steinmeier/tabatabai:2007} estimate that around one fourth of the survey respondents misclassified their pension plan type. 

The dataset in my analysis is from the Survey of Income and Program Participation (SIPP) of 1991. It has been used in various analyses, e.g., \cite{poterba/venti/wise:1995} and \cite{abadie:2003}.
I follow the data construction in \cite{abadie:2003}. 
The sample consists of households in which at least one person is employed, which has no income from self-employment, and whose annual family income is between \$10,000 and \$200,000.
The resulting sample size is 9,275.

The outcome variable $Y$ is the net financial assets, the measured treatment variable $T$ is the self-reported participation in 401(k), $Z$ is the eligibility for 401(k) and $R$ is the participation in an individual retirement account (IRA). The control variables $V$ includes constant, family income, age and its square, marital status, and family size. I compute the summary statistics for these variables in Table \ref{table_summary}. 
The 401(k) participation can be endogenous, because participants in 401(k) might be more informed or plan more about retirement savings than non-participants. 
To control for the endoeneity problem, this paper uses 401(k) eligibility as an instrumental variable.

I use the linear probability model for the regression of the instrumental variable $Z$ on the control variables $V$, that is, 
$E[Z\mid V]=\pi(V)=V'\beta_{\pi}$. 
For a comparison purpose, I compute the Wald estimator, $\$16,290$, with a 95\% bootstrapped confidence interval $[5,976,\  27,611]$.\footnote{The Wald estimator here is the sample analogue of $E\left[\frac{Z-\pi(V)}{\pi(V)(1-\pi(V))}Y\right]/E\left[\frac{Z-\pi(V)}{\pi(V)(1-\pi(V))}T\right]$ with $\pi$ being estimated in the linear probability model.}
The intent-to-treat effect $\Delta E[Y\mid Z]$ is estimated as  $\$10,981$ with  a 95\% bootstrapped confidence interval $[4,169,\  18,558]$

Table \ref{table_empri1}-\ref{table_empri3} shows that the confidence intervals for the local average treatment effect, under different assumptions (Theorem \ref{theorem1}, \ref{theorem4}, \ref{theorem1-less}, respectively).\footnote{I use 2000 draws for the bootstrap and set $\beta=0.1\%$ for the moment selection and $\delta=1\%$ for the estimation of $\beta_{\pi}$. See Appendix A for the multiplier bootstrap critical value.}
The confidence intervals in these tables are robust to a misclassification of the treatment variable. 
They are wider than the 95\% confidence interval $[5,976,\  27,611]$ for the Wald estimator, but are in general comparable to the confidence interval for the Wald estimator.
The confidence intervals in this exercise do not shrink as $K_n$ increases from $1$ to $5$. (Note that, when $K_n=1$, the moment inequalities in (\ref{Approx3}) is equivalent to using the Wald estimand as the upper bound for $|\theta|$.)
This is possibly because the data generation process does not violate the conditions in (\ref{TESTABLEIMPLICAT}) to a large extent, and therefore the Wald estimand is close (even if not equal) to the sharp upper bound for the local average treatment effect. 

Table \ref{table_empri2} summarizes the confidence intervals with the IRA participation $R$ as an additional measurement, as discussed in Theorem \ref{theorem4}. 
It shows similar values to Table \ref{table_empri1} and it can be interpreted as a result that the IRA participation $R$ has only little identifying power on the local average treatment effect in this empirical exercise. 

Table \ref{table_empri3} summarizes the confidence intervals without using the measured treatment $T$, as in Theorem \ref{theorem1-less}. 
The lower bound of the confidence intervals does not change from those in Table \ref{table_empri1}, because the lower bound of the identified set does not change without the information from the measured treatment $T$. 
The upper bound is 3-4 times larger than those in Table \ref{table_empri1}, which is the cost of not using $T$.\footnote{When $K_n=1$, $TV_Y$ becomes zero and then there is no finite upper bound on the local average treatment effect.}

\begin{table}
\centering
\begin{tabular}{l|cc}
\hline 
\hline 
&mean&standard deviation\\
\hline
$Y$: family net financial assets & 19,071 &  63,963\\
$T$: 401(k) participation &     0.2762  &  0.4472\\
$Z$: IRA participation &     0.3921  &  0.4883\\
$R$: 401(k) eligibility &     0.2543  &  0.4355\\
    \hline
\end{tabular}
\caption{Summary statistics for $Y$, $T$, and $Z$}
\label{table_summary}
%\end{table}
%\begin{table}
%\centering
\bigskip\bigskip
\begin{tabular}{c|rlrl}
\hline 
$K_n$ & \multicolumn{2}{c}{90\% CI}  & \multicolumn{2}{c}{95\% CI} \\
\hline
1 & [4,899 &   26,741]  &  [3,626&   28,945]\\
2 & [4,883 &   26,994]  &  [3,619&   29,077]\\
3 & [4,860 &  27,002]   & [3,596&   29,124]\\
4 & [4,851 &   27,171]  &  [3,595&   29,302]\\
    \hline
\end{tabular}
\caption{90\% and 95\% confidence intervals for the local average treatment effect for different $K_n$. (Based on Theorem \ref{theorem1})}
\label{table_empri1}
%\end{table}
%\begin{table}
%\centering
\bigskip\bigskip
\begin{tabular}{c|rlrl}
\hline 
$K_n$ & \multicolumn{2}{c}{90\% CI}  & \multicolumn{2}{c}{95\% CI} \\
\hline
1 & [4,899 &  26,994]  &  [3,619&    29,241]\\
2 & [4,859 &   27,148]  &  [3,602&   29,273]\\
3 & [4,830 &  27,396]   & [3,590&   29,331]\\
4 & [4,832 &   27,581]  &  [3,582&   29,488]\\
    \hline
\end{tabular}
\caption{90\% and 95\% confidence intervals with using $R$ as in Theorem \ref{theorem4}}
\label{table_empri2}
%\end{table}
%\begin{table}
%\centering
\bigskip\bigskip
\begin{tabular}{c|rlrl}
\hline 
$K_n$ & \multicolumn{2}{c}{90\% CI}  & \multicolumn{2}{c}{95\% CI} \\
\hline
1 & [4,912 &   Inf]  &  [3,647&   Inf]\\
2 & [4,889 &  79,281 ]  &  [3,625&   85,330]\\
3 & [4,876 &  102,242]   & [3,614&   109,829]\\
4 & [4,878 &   123,440]  &  [3,618&   132,642]\\
    \hline
\end{tabular}
\caption{90\% and 95\% confidence intervals without $T$ as in Theorem \ref{theorem1-less}}
\label{table_empri3}
\end{table}

\newpage
\section{Numerical example and Monte Carlo simulations}\label{sec5}
This section considers a numerical example to illustrates the theoretical properties in the previous section. 
I consider the following data generating process: 
\begin{eqnarray*}
Z&\sim&Bernoulli(0.5)\\
T^\ast&=&1\{-3/4+1/2Z+U_1\geq 0\}\\
Y&=&2T^\ast+\Phi(U_2)\\
T&=&T^\ast+(1-2T^\ast)1\{U_3\leq\gamma\}
\end{eqnarray*}
where $\Phi$ is the standard normal cdf and, conditional on $Z$, $(U_1,U_2,U_3)$ is drawn from the Gaussian copula with the correlation matrix $$\left(\begin{array}{ccc}1&0.25&0.25\\ 0.25&1&0.25\\ 0.25&0.25&1\end{array}\right).$$ 
I set $\gamma=0, 0.2, 0.4$, which captures the degree of the misclassification. 
In this design, the treatment variable is endogenous since $U_1$ and $U_2$ are correlated.  
In addition, the misclassification is endogenous in that $U_2$ and $U_3$ are correlated. 

Table \ref{table_poppara}  lists the three population objects: the local average treatment effect, the Wald estimand, and the identified set for the local average treatment effect. 
Note that, unless $\gamma=0$, the distribution for $(Y,T,Z)$ violates the conditions in (\ref{TESTABLEIMPLICAT}).
When there is no measurement error, the sharp upper bound is equal to the Wald estimand, which is the case for $\gamma=0$.  
When there is a measurement error, the sharp upper bound for the local average treatment effect can be smaller than the Wald estimand. 

In order to focus on the finite sample properties of the test $\mathbbm{1}\{T(\theta,\pi)>c(\alpha,\theta,\pi)\}$, I only evaluate coverage probabilities given $\pi=0.5$ for various value of $\theta$.   
The partition of grids is equally spaced over $\mathbf{Y}$ with the number the partitions $K_n=1,\ldots,4$.   
Coverage probabilities are calculated as how often the $95\%$ confidence interval includes a given parameter value out of $1000$ simulations.
The sample size is $n=500$ for Monte Carlo simulations. I use $1000$ bootstrap repetitions to construct critical values.
I set $\beta=0.1\%$ for the moment selection. 

Figures \ref{fig1}-\ref{fig3} describe the coverage probabilities of the confidence intervals for each parameter value. 
When the degree of measurement error is zero ($\gamma=0$), the power for the confidence interval with $K_n=1$ has a slightly better performance than those with $K_n\geq 2$. It can be because the number of moment inequalities are larger for $K_n\geq 2$ and then the critical value is bigger. 
As the degree of measurement error becomes larger, the power for the confidence intervals with $K_n\geq 2$ becomes better than that with $K_n=1$. It is a result of the fact that the sharp upper bound for the local average treatment effect is smaller than the Wald estimand. 

Next, I investigate the identifying power of an additional measurement. 
$$
R=T^\ast+(1-2T^\ast)1\{U_4\leq\gamma\}
$$
where $(U_1,U_2,U_3,U_4)$ is drawn from the Gaussian copula with the correlation matrix 
$$\left(\begin{array}{cccc}1&0.25&0.25&0.25\\ 0.25&1&0.25&0.25\\ 0.25&0.25&1&0.25\\ 0.25&0.25&0.25&1\end{array}\right).$$ 
Table \ref{table_poppara2}  lists the three population objects: the local average treatment effect, the Wald estimand, and the identified set for the local average treatment effect. 
Figures \ref{fig4}-\ref{fig6} describe the coverage probabilities of the confidence intervals for each parameter value. 
The comparison among different $K_n$'s are similar to the previous figures.  

Last, I consider the dependence between measurement error and instrumental variable, as in Section \ref{sec3Less}.
Table \ref{table_poppara3} lists the three population objects and Figures \ref{fig7}-\ref{fig9} describe the coverage probabilities of the confidence intervals.
Since they do not use any information from the measured treatment $T$, the identified sets and the confidence intervals show that the upper bounds on the local average treatment effect is larger than those under the independence between measurement error and instrumental variable. 
The difference becomes smaller when the degree of the measurement error is larger.  
It can be considered as the result that, when the misclassification happens too often, the measured treatment $T$ has only little information about the true treatment and therefore there is a small difference between the identified sets.

\begin{table}
\centering
\begin{tabular}{l|ccc}
\hline 
\hline 
$\gamma$ & LATE & Identified set & Wald estimand \\
\hline
$0$ 		& 2.00 & [1,\ 2.00] & 2.00 \\
$0.2$ 	& 2.00 & [1,\ 2.41] & 3.01 \\
$0.4$ 	& 2.00 & [1,\ 2.64] & 8.72\\
    \hline
\end{tabular}
\caption{Population parameters (Theorem \ref{theorem1})}
\label{table_poppara}
%\end{table}
%\begin{table}
%\centering

\bigskip\bigskip
\begin{tabular}{l|ccc}
\hline 
\hline 
$\gamma$ & LATE & Identified set & Wald estimand \\
\hline
$0$ 		& 2.00 & [1,\ 2.00] & 2.00 \\
$0.2$ 	& 2.00 & [1,\ 2.26] & 3.01 \\
$0.4$ 	& 2.00 & [1,\ 2.62] & 8.72\\
    \hline
\end{tabular}
\caption{Population parameters  with using $R$ (Theorem \ref{theorem4})}
\label{table_poppara2}
%\end{table}
%\begin{table}
%\centering
\bigskip\bigskip
\begin{tabular}{l|ccc}
\hline 
\hline 
$\gamma$ & LATE & Identified set & Wald estimand \\
\hline
$0$ 		& 2.00 & [1,\ 2.67] & 2.00 \\
$0.2$ 	& 2.00 & [1,\ 2.67] & 3.01 \\
$0.4$ 	& 2.00 & [1,\ 2.68] & 8.72\\
    \hline
\end{tabular}
\caption{Population parameters without $T$ (Theorem \ref{theorem1-less}) }
\label{table_poppara3}
\end{table}
\begin{figure}
\centering
\includegraphics[width=.8\textwidth,keepaspectratio]{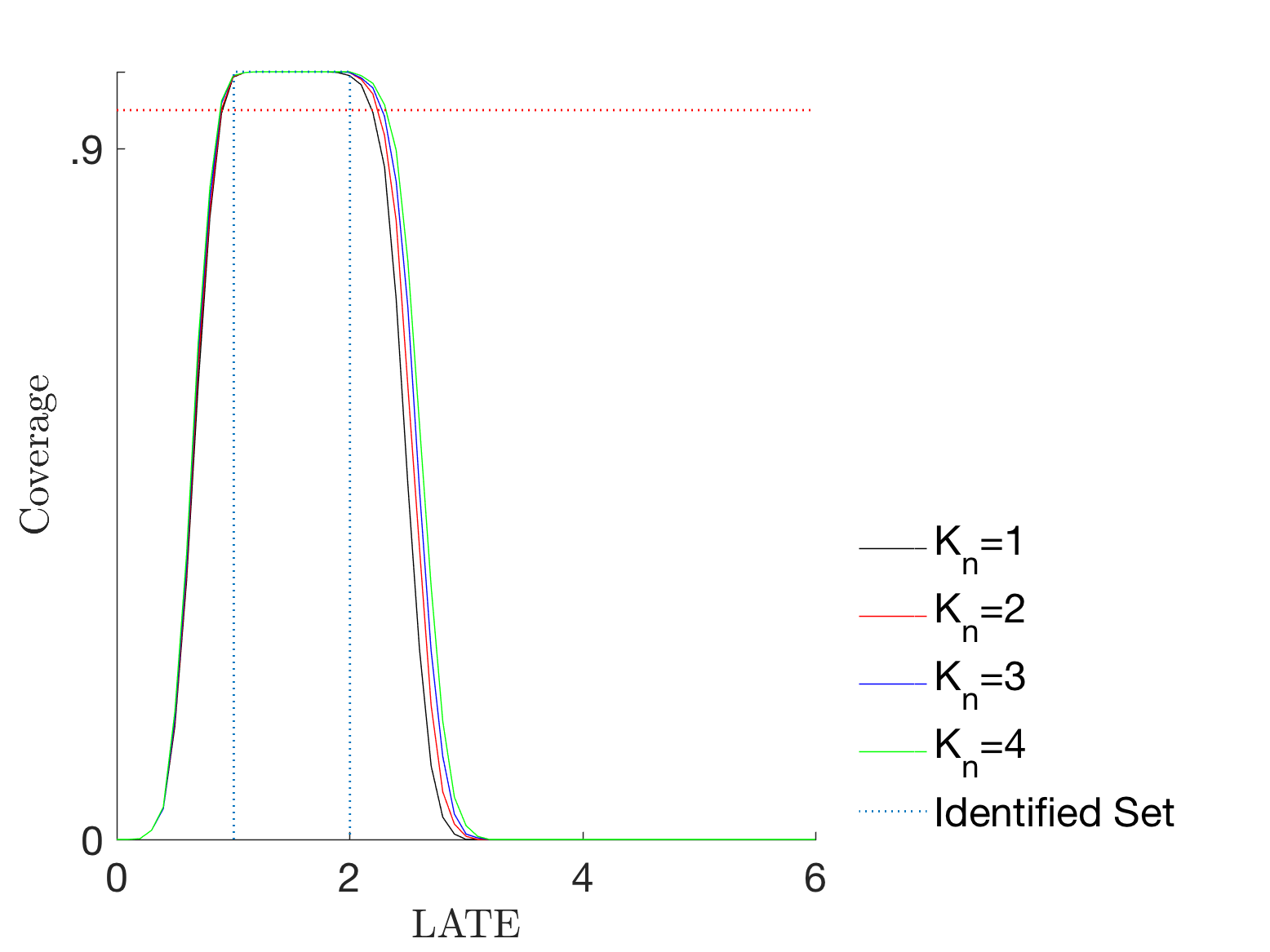}
\caption{Coverage of the confidence interval (Theorem \ref{theorem1}) for $\gamma=0$.}
\label{fig1}
\end{figure}
\begin{figure}
\centering
\includegraphics[width=.8\textwidth,keepaspectratio]{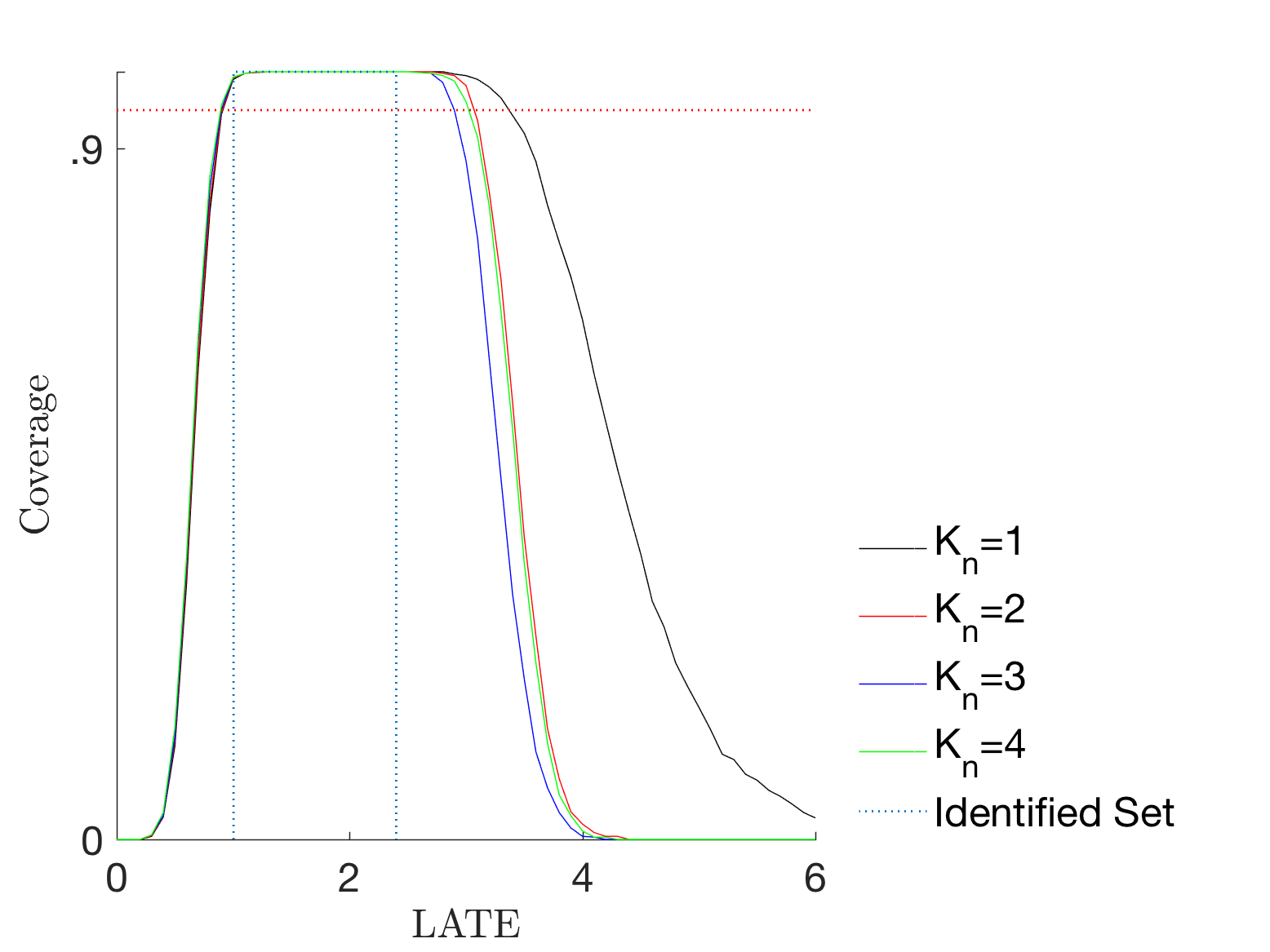}
\caption{Coverage of the confidence interval (Theorem \ref{theorem1}) for $\gamma=0.2$.}
\label{fig2}
\end{figure}
\begin{figure}
\centering
\includegraphics[width=.8\textwidth,keepaspectratio]{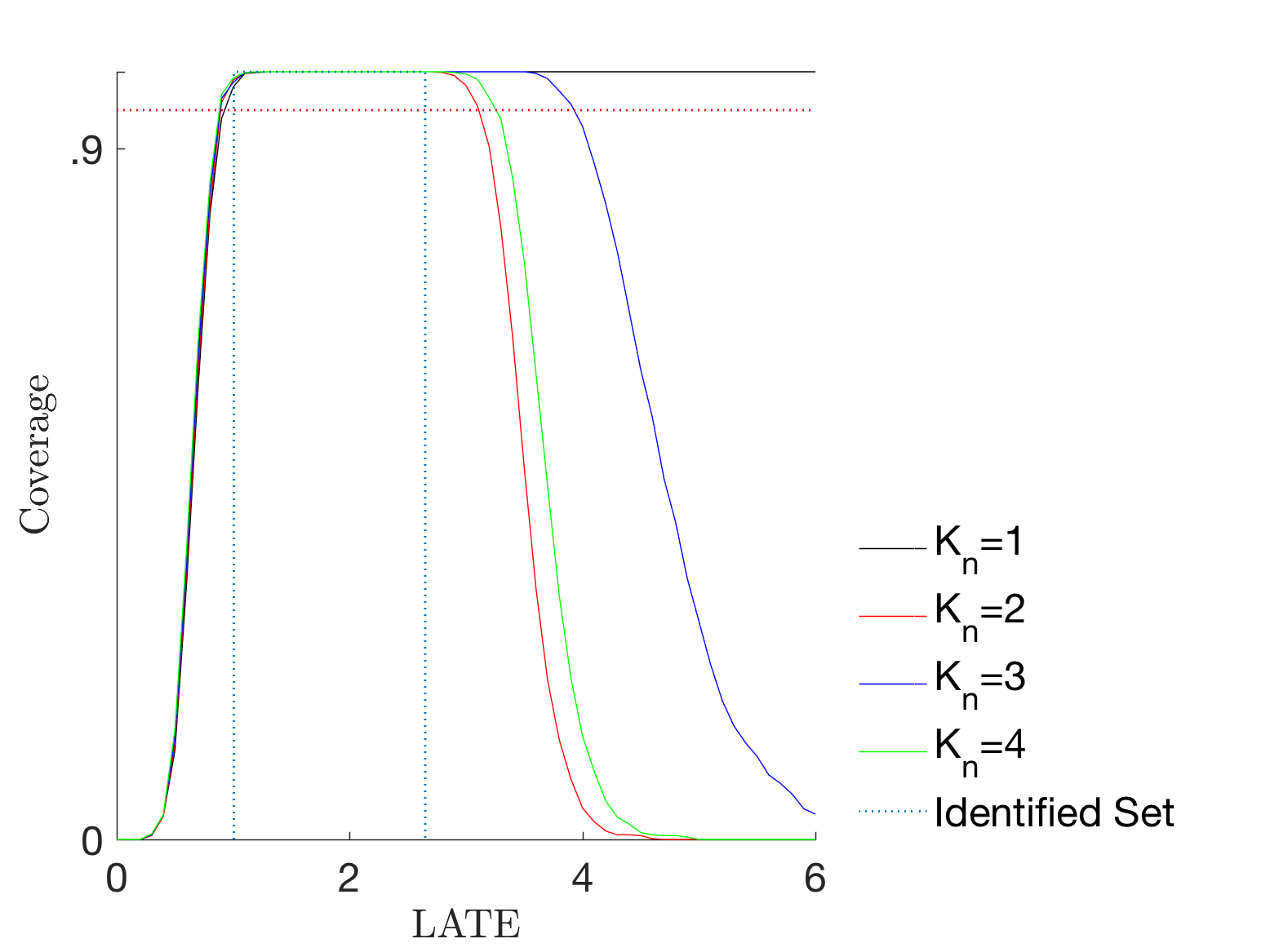}
\caption{Coverage of the confidence interval (Theorem \ref{theorem1}) for $\gamma=0.4$.}
\label{fig3}
\end{figure}
\begin{figure}
\centering
\includegraphics[width=.8\textwidth,keepaspectratio]{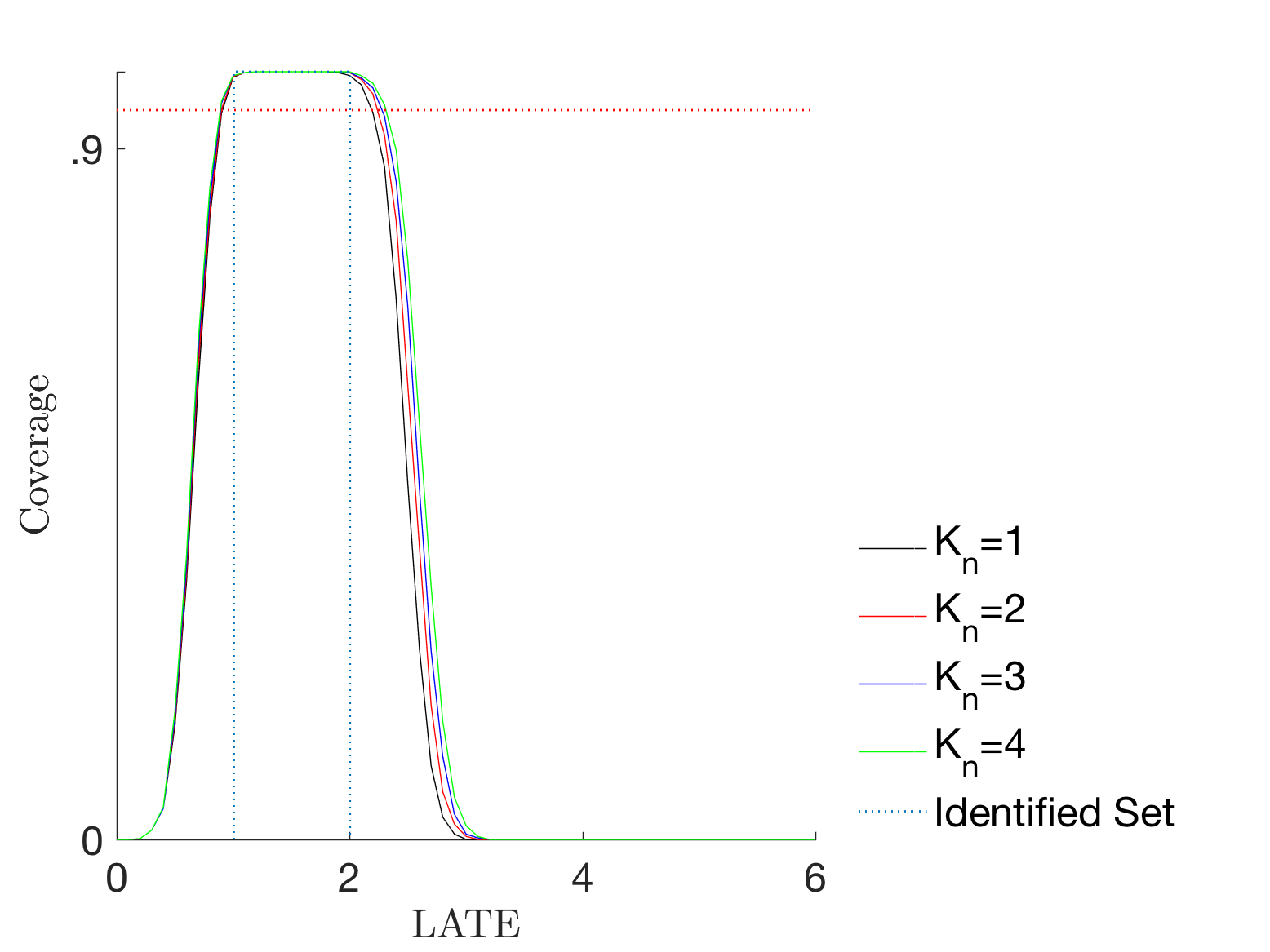}
\caption{Coverage of the confidence interval with using $R$ (Theorem \ref{theorem4}) for $\gamma=0$.}
\label{fig4}
\end{figure}
\begin{figure}
\centering
\includegraphics[width=.8\textwidth,keepaspectratio]{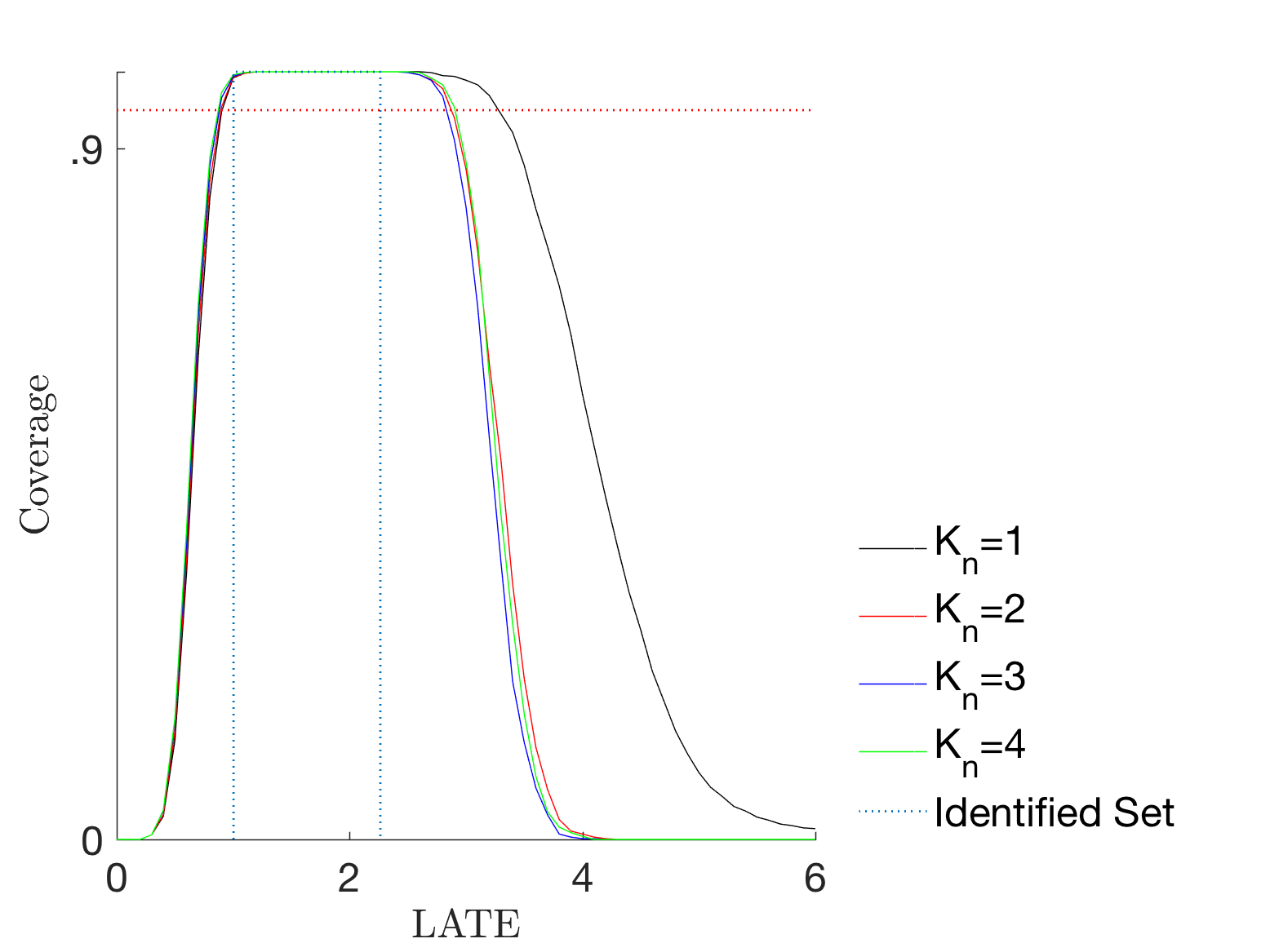}
\caption{Coverage of the confidence interval  with using $R$ (Theorem \ref{theorem4}) for $\gamma=0.2$.}
\label{fig5}
\end{figure}
\begin{figure}
\centering
\includegraphics[width=.8\textwidth,keepaspectratio]{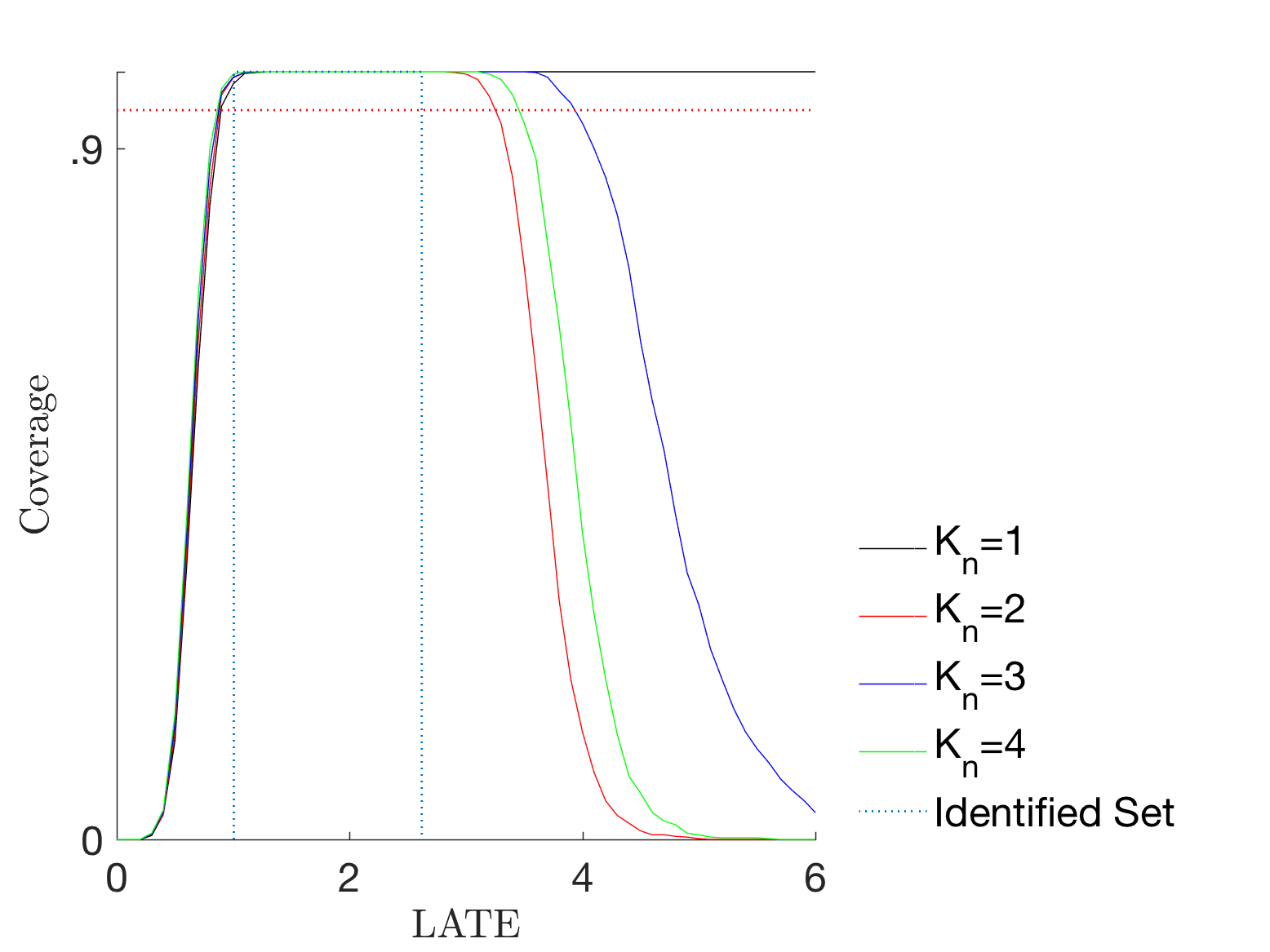}
\caption{Coverage of the confidence interval  with using $R$ (Theorem \ref{theorem4}) for $\gamma=0.4$.}
\label{fig6}
\end{figure}
\begin{figure}
\centering
\includegraphics[width=.8\textwidth,keepaspectratio]{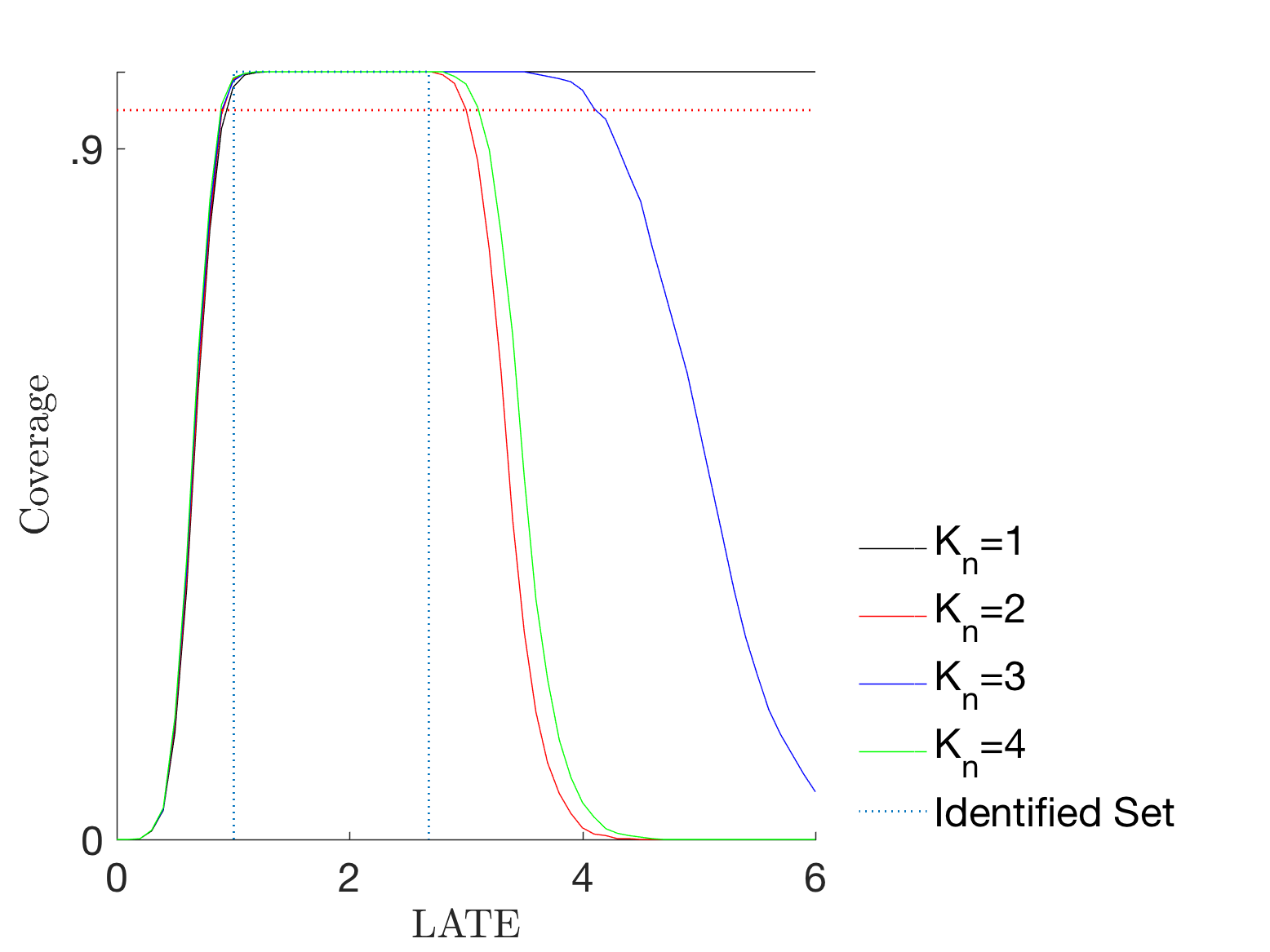}
\caption{Coverage of the confidence interval  without $T$ (Theorem \ref{theorem1-less}) for $\gamma=0$.}
\label{fig7}
\end{figure}
\begin{figure}
\centering
\includegraphics[width=.8\textwidth,keepaspectratio]{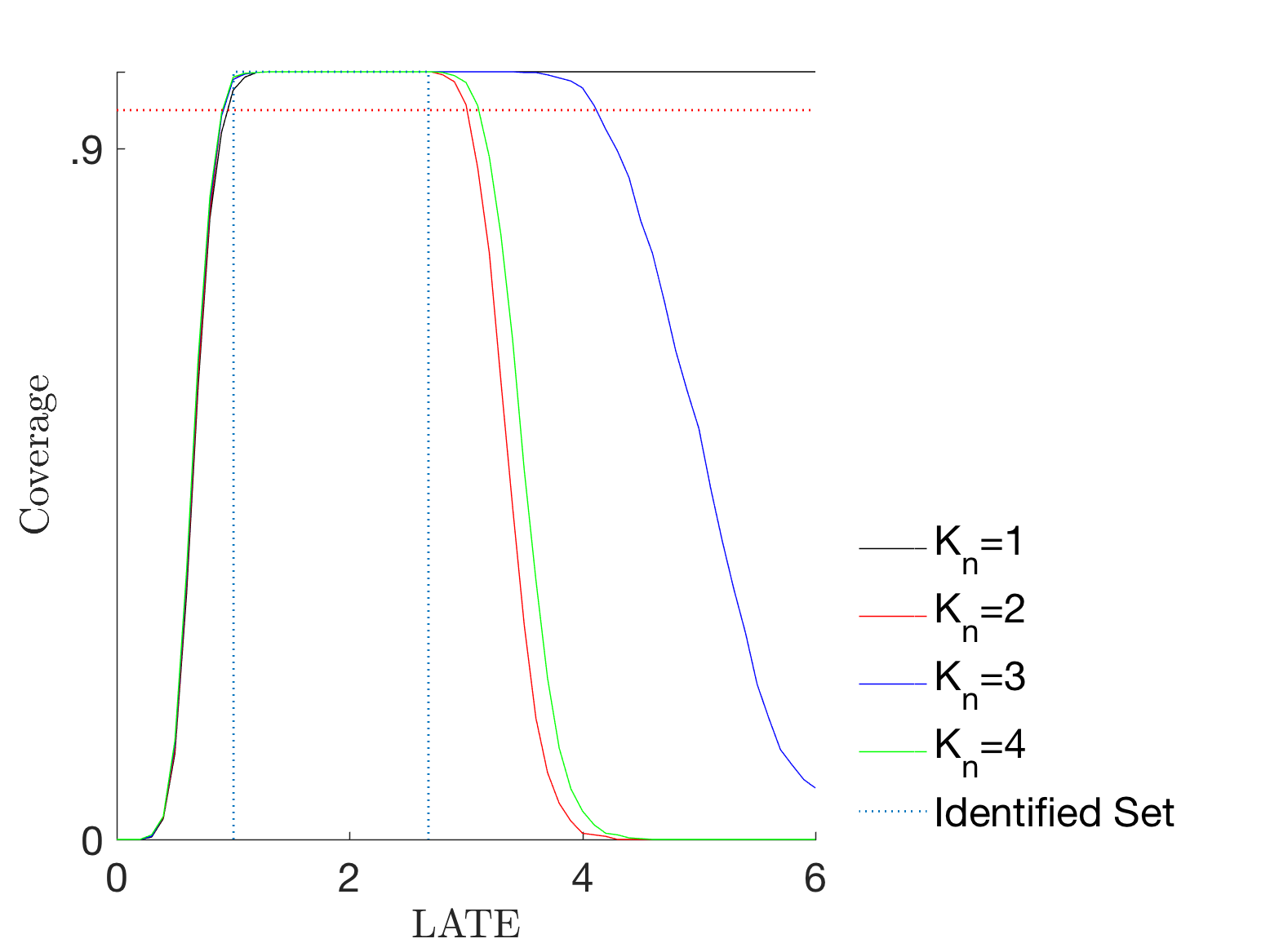}
\caption{Coverage of the confidence interval  without $T$ (Theorem \ref{theorem1-less}) for $\gamma=0.2$.}
\label{fig8}
\end{figure}
\begin{figure}
\centering
\includegraphics[width=.8\textwidth,keepaspectratio]{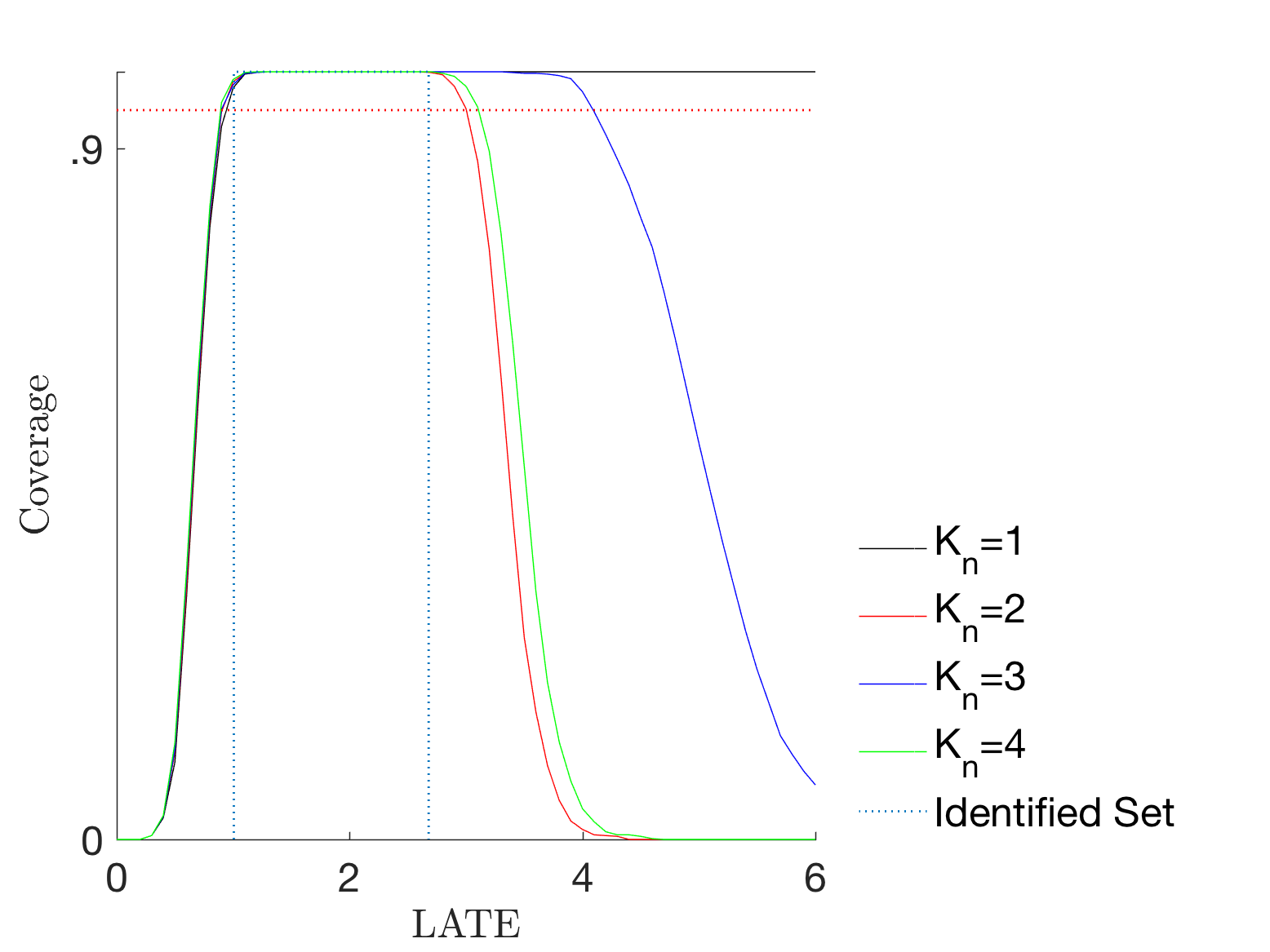}
\caption{Coverage of the confidence interval without $T$ (Theorem \ref{theorem1-less}) for $\gamma=0.4$.}
\label{fig9}
\end{figure}

\newpage
\section{Conclusion}\label{sec6}
This paper studies the identifying power of an instrumental variable in the heterogeneous treatment effect framework when a binary treatment is mismeasured and endogenous. 
The assumptions in this framework are the monotonicity of the instrumental variable $Z$ on the true treatment $T^\ast$ and the exogeneity of $Z$.  
I use the total variation distance to characterize the identified set for the local average treatment effect $E[Y_1-Y_0\mid T_0^\ast<T_1^\ast]$. 
I also provide an inference procedure for the local average treatment effect.
Unlike the existing literature on measurement error, the identification strategy does not reply on a specific structure of the measurement error; the only assumption on the measurement error is its independence of the instrumental variable.

There are several directions for future research. 
First, the choice of the partition $\mathbf{I}_n$ in Section \ref{sec4},  particularly the choice of $K_n$, is an interesting direction. 
To the best of my knowledge, the literature on many moment inequalities has not investigated how econometricians choose the numbers of the many moment inequalities, e.g., \cite{andrews/shi:2016}.  
Second, it may be worthwhile to investigate the other parameter for the treatment effect. 
This paper has focused on the local average treatment effect, but the literature on heterogeneous treatment effect has emphasized the importance of choosing a adequate treatment effect parameter in order to answer relevant policy questions. 
Third, it is also interesting to investigate various assumptions on the measurement errors. 
In some empirical settings, for example, it may be reasonable to assume that the measurement error is one-directional (e.g., misclassification happens only when $T^\ast=1$). 
Fourth, it is not trivial how the analysis of this paper can be extended to an instrumental variable taking more than two values. For a general instrumental variable, it is always possible to focus on two values of the instrumental variable and apply the analysis of this paper to the subpopulation with the instrumental variable taking these two values.  
However, different pairs of the values can have different compliers, so that the parameter of interest is not common across different pairs, as in \cite{heckman/vytlacil:2005}.  

\appendix
\section*{Appendix A: Multiplier bootstrap}\label{CCKbootstrap}
This section describes the multiplier bootstrap in \cite{chernozhukov/chetverikov/kato:2014}. For the sake of simplicity, I do not talk about their two- and three-step versions of the multiplier bootstrap.

Define the moment functions based on (\ref{Approx1})-(\ref{Approx3}). Define $\mathbf{H}_n=\{h_1,\cdots,h_{p_n-2}\}$ and 
\begin{eqnarray*}
g_1(W,\theta,\pi)&=&-\frac{Z-\pi(V)}{\pi(V)(1-\pi(V))}\mathrm{sgn}(\theta)Y\\
g_2(W,\theta,\pi)&=&\frac{Z-\pi(V)}{\pi(V)(1-\pi(V))}\mathrm{sgn}(\theta)Y-|\theta|\\
g_{2+j}(W,\theta,\pi)&=&\frac{Z-\pi(V)}{\pi(V)(1-\pi(V))}\left(|\theta|h_j(Y,T,V)-\mathrm{sgn}(\theta)Y\right).
\end{eqnarray*}
Then the approximated identified set is characterized by 
$$
\{(\theta,\pi)\in\Theta\times[0,1]: E_P[g_j(W,\theta,\pi)]\mbox{ for every }j=1,\ldots,p_n\}.
$$

I construct a confidence interval $\mathcal{C}_{\theta,n}(\alpha)$ via the multiplier bootstrap in \cite{chernozhukov/chetverikov/kato:2014}.
The test statistic for the true parameter values being $\psi$ is defined by 
$$
T(\theta,\pi)=\max_{1\leq j\leq p_n}\frac{\sqrt{n}\hat{m}_j(\theta,\pi)}{\hat\sigma_j(\theta,\pi)}, 
$$
where $\hat{m}_j(\theta,\pi)$ estimates $m_j(\theta,\pi)= E_P[g_j(W,\theta,\pi)]$, and $\hat\sigma_j^2(\theta,\pi)$ estimates the variance $\sigma_j^2(\theta,\pi)$ of $\sqrt{n}\hat{m}_j(\theta,\pi)$:
\begin{eqnarray*}
\hat{m}_j(\theta,\pi)&=& n^{-1}\sum_{i=1}^ng_j(W_i,\theta,\pi)\\
\hat\sigma_j^2(\theta,\pi)&=& n^{-1}\sum_{i=1}^n(g_j(W_i,\theta,\pi)-\hat{m}_j(\theta,\pi))^2.
\end{eqnarray*}
To conduct a multiplier bootstrap, generate $n$ independent standard normal random variables $\epsilon_1,\ldots,\epsilon_n$. 
The centered bootstrap moments are 
$$
\hat{m}_j^{B}(\theta,\pi)= n^{-1}\sum_{i=1}^n\epsilon_i\left(g_j(W_i,\theta,\pi)-\hat{m}_j(\theta,\pi)\right).
$$
The bootstrapped test statistic is defined by  
$$
T^B(\theta,\pi)=\max_{1\leq j\leq p_n}\sqrt{n}\frac{\hat{m}_j^{B}(\psi)}{\hat\sigma_j(\psi)}.
$$
The critical value $c(\alpha,\theta,\pi)$ is defined as the conditional $(1-\alpha)$-quantile of $T^B(\theta,\pi)$ given $\{W_i\}$: 
$$
Pr\left(T^B(\theta,\pi)\leq c(\alpha,\theta,\pi)\mid \{W_i\}\right)=1-\alpha.
$$

\section*{Appendix B: Proofs}
\subsection*{Proof of Lemma \ref{lemma1}}
By Equation (\ref{eqIA}), 
$\theta(P^{\ast})\Delta E_{P^{\ast}}[Y\mid Z]=\theta(P^{\ast})^2P^{\ast}(T_0^\ast<T_1^\ast)\geq 0$, and $|\Delta E_{P^{\ast}}[Y\mid Z]|=|\theta(P^{\ast})P^{\ast}(T_0^\ast<T_1^\ast)|\leq|\theta(P^{\ast})|$.

\subsection*{Proof of Lemma \ref{lemma2}}
I obtain $f_{(Y,T)\mid Z=1}-f_{(Y,T)\mid Z=0}=P^{\ast}(T_0^\ast<T_1^\ast)(f_{(Y_1,T_1)\mid T_0^\ast<T_1^\ast}-f_{(Y_0,T_0)\mid T_0^\ast<T_1^\ast})$ by the same logic as Theorem 1 in \cite{imbens/angrist:1994}: 
\begin{eqnarray*}
f_{(Y,T)\mid Z=0}
&=&
P^{\ast}(T_0^\ast=T_1^\ast=1\mid Z=0)f_{(Y,T)\mid Z=0,T_0^\ast=T_1^\ast=1}
\\&&\quad+P^{\ast}(T_0^\ast<T_1^\ast\mid Z=0)f_{(Y,T)\mid Z=0,T_0^\ast<T_1^\ast}
\\&&\quad+P^{\ast}(T_0^\ast=T_1^\ast=0\mid Z=0)f_{(Y,T)\mid Z=0,T_0^\ast=T_1^\ast=0}\\
&=&
P^{\ast}(T_0^\ast=T_1^\ast=1)f_{(Y_1,T_1)\mid T_0^\ast=T_1^\ast=1}
\\&&\quad+P^{\ast}(T_0^\ast<T_1^\ast)f_{(Y_0,T_0)\mid T_0^\ast<T_1^\ast}
\\&&\quad+P^{\ast}(T_0^\ast=T_1^\ast=0)f_{(Y_0,T_0)\mid T_0^\ast=T_1^\ast=0}\\
f_{(Y,T)\mid Z=1}
&=&
P^{\ast}(T_0^\ast=T_1^\ast=1)f_{(Y_1,T_1)\mid T_0^\ast=T_1^\ast=1}
\\&&\quad+P^{\ast}(T_0^\ast<T_1^\ast)f_{(Y_1,T_1)\mid T_0^\ast<T_1^\ast}
\\&&\quad+P^{\ast}(T_0^\ast=T_1^\ast=0)f_{(Y_0,T_0)\mid T_0^\ast=T_1^\ast=0}.
\end{eqnarray*}
By the triangle inequality, 
\begin{eqnarray*}
TV_{(Y,T)}
&=&
\frac{1}{2}\sum_{t=0,1}\int |f_{(Y,T)\mid Z=1}(y,t)-f_{(Y,T)\mid Z=0}(y,t)|d\mu_Y(y)\\
&=&
\frac{1}{2}\sum_{t=0,1}\int |P^{\ast}(T_0^\ast<T_1^\ast)(f_{(Y_1,T_1)\mid T_0^\ast<T_1^\ast}(y,t)-f_{(Y_0,T_0)\mid T_0^\ast<T_1^\ast}(y,t))|d\mu_Y(y)\\
&=&
P^{\ast} (T_0^\ast<T_1^\ast)\frac{1}{2}\sum_{t=0,1}\int |f_{(Y_1,T_1)\mid T_0^\ast<T_1^\ast}(y,t)-f_{(Y_0,T_0)\mid T_0^\ast<T_1^\ast}(y,t)|d\mu_Y(y)\\
&\leq&
P^{\ast} (T_0^\ast<T_1^\ast)\frac{1}{2}\sum_{t=0,1}\int (f_{(Y_1,T_1)\mid T_0^\ast<T_1^\ast}(y,t)+f_{(Y_0,T_0)\mid T_0^\ast<T_1^\ast}(y,t))d\mu_Y(y)\\
&=&
P^{\ast} (T_0^\ast<T_1^\ast).
\end{eqnarray*}
Moreover, since $T_0^\ast\leq T_1^\ast$ almost surely, 
\begin{eqnarray*}
P^{\ast} (T_0^\ast<T_1^\ast),
&=&
\frac{1}{2}|P^{\ast} (T_0^\ast=1)-P^{\ast} (T_1^\ast=1)|+\frac{1}{2}|P^{\ast} (T_0^\ast=0)-P^{\ast} (T_1^\ast=0)|\\
&=&
\frac{1}{2}\sum_{t^\ast=0,1}|f_{T^{\ast}\mid Z=1}(t^\ast)-f_{T^{\ast}\mid Z=0}(t^\ast)|\\
&=&
TV_{T^\ast}.
\end{eqnarray*}

\subsection*{Proof of Lemma \ref{Wald_lemma}}
Based on the triangle inequality, 
\begin{eqnarray*}
TV_{(Y,T)}
&=&
\frac{1}{2}\sum_{t=0,1}\int |f_{(Y,T)\mid Z=1}(y,t)-f_{(Y,T)\mid Z=0}(y,t)|d\mu_Y(y)\\
&\geq&
\frac{1}{2}\sum_{t=0,1}|\int(f_{(Y,T)\mid Z=1}(y,t)-f_{(Y,T)\mid Z=0}(y,t))d\mu_Y(y)|\\
&=&
\frac{1}{2}\sum_{t=0,1}|f_{T\mid Z=1}(t)-f_{T\mid Z=0}(t)|\\
&=&
|f_{T\mid Z=1}(1)-f_{T\mid Z=0}(1)|\\
&=&
|\Delta E_P[T\mid Z]|.
\end{eqnarray*}
The equality holds if and only if the sign of $f_{(Y,T)\mid Z=1}(y,t)-f_{(Y,T)\mid Z=0}(y,t)$ is constant in $y$ for every $t$. 
Since $|\Delta E_P[T\mid Z]|=\Delta E_P[T\mid Z]$ if and only if $f_{T\mid Z=1}(1)-f_{T\mid Z=0}(1)$ is positive, the condition in (\ref{TESTABLEIMPLICAT}) is a necessary and sufficient condition for $TV_{(Y,T)}=\Delta E_P[T\mid Z]$, which is equivalent for the Wald estimand to belong to the identified set. 

\subsection*{Proof of Theorems \ref{theorem1}, \ref{theorem1conditional} and \ref{theorem4}}
Theorems \ref{theorem1}, \ref{theorem1conditional} and \ref{theorem4} are special cases of the following theorem. 
\begin{theorem}\label{intergreated_thereom1}
There is some variable $V$ taking values in a set $\mathbf{V}$ satisfying the following properties. 
\begin{itemize}
\item[(i)] For each $t^\ast=0,1$, $Z$ is conditionally independent of $(R_{t^\ast},T_{t^\ast},Y_{t^\ast},T^\ast_{0},T^\ast_{1})$ given $V$.
\item[(ii)] $T^\ast_{1}\geq T^\ast_{0}$ almost surely.
\item[(iii)] $0<P(Z=1\mid V)<1$. 
\end{itemize} 
Consider an arbitrary data distribution $P$ of $(R,Y,T,Z,V)$. 
The identified set $\Theta_I(P)$ for the local average treatment effect is the set of $\theta\in\Theta$ satisfying the following three inequalities. 
\begin{eqnarray*}
&&|\theta|E_P[TV_{(R,Y,T)\mid V}]\leq |E_P[\Delta E_P[Y\mid Z,V]]|\\
&&\theta E_P[\Delta E_P[Y\mid Z,V]]\geq 0\\
&&|\theta|\geq |E_P[\Delta E_P[Y\mid Z,V]]|.
\end{eqnarray*} 
\end{theorem}

To show Theorem \ref{intergreated_thereom1}, I consider the two cases separately: $E_P[TV_{(R,Y,T)\mid V}]=0$ or $E_P[TV_{(R,Y,T)\mid V}]>0$.
\subsubsection*{Case 1: $E_P[TV_{(R,Y,T)\mid V}]=0$.}
In this case, $f_{(R,Y,T)\mid Z=z,V}=f_{(R,Y,T)\mid V}$ a.s. for every $z=0,1$.
Let $f_1$ and $f_0$ be any pair of density functions for $Y$ dominated by $\mu_Y$. 
Define the data generating process $P_{f_1,f_0}^{\ast}$: 
\begin{eqnarray*}
&(Z,V)&\sim\ f_{(Z,V)}\\ 
&(T^\ast_{0},T^\ast_{1})\mid Z,V&=\ (1,1)\\
&(R_1,Y_1,T_1)\mid (T^\ast_{0},T^\ast_{1},Z,V)&\sim\
\begin{cases}
f_{(R,Y,T)}(r,y,t)&\mbox{ if }T^\ast_{0}=T^\ast_{1}\\
f_1(y)f_{(R,T)}(r,t)&\mbox{ if }T^\ast_{0}\ne T^\ast_{1}
\end{cases}\\
&(R_0,Y_0,T_0)\mid (T^\ast_{0},T^\ast_{1},Z)&\sim\ f_0(y)f_{(R,T)}(r,t).
\end{eqnarray*}

Theorem \ref{intergreated_thereom1} follows from the following three observations: 
\begin{itemize}
\item[(i)] $P_{f_1,f_0}^{\ast}$ satisfies all the assumptions in Theorem \ref{intergreated_thereom1}; 
\item[(ii)] $P_{f_1,f_0}^{\ast}$ generates the data distribution $P$; 
\item[(iii)] under $P_{f_1,f_0}^{\ast}$, the local average treatment effect is $\int y(f_1(y)-f_0(y))d\mu_Y(y)$.
\end{itemize}

(i) $P_{f_1,f_0}^{\ast}$ satisfies the independence between $Z$ and $(R_{t^\ast},T_{t^\ast},Y_{t^\ast},T^\ast_{0},T^\ast_{1})$ given $V$ for each $t^\ast=0,1$. 
Furthermore, $P_{f_1,f_0}^{\ast}$ satisfies $T^\ast_{1}\geq T^\ast_{0}$ almost surely. 

(ii) 
Denote by $f^{\ast}$ the density function of $P_{f_1,f_0}^{\ast}$.
Then 
\begin{eqnarray*}
f^{\ast}_{(R,Y,T)\mid Z=0,V}
&=&
f^{\ast}_{(R_1,Y_1,T_1)\mid Z=0,V}\\
&=&
f_{(R,Y,T)\mid V}\\
&=&
f_{(R,Y,T)\mid Z=0,V}\\
f^{\ast}_{(R,Y,T)\mid Z=1,V}
&=&
f^{\ast}_{(R_1,Y_1,T_1)\mid Z=0,V}\\
&=&
f_{(R,Y,T)\mid V}\\
&=&
f_{(R,Y,T)\mid Z=1,V}.
\end{eqnarray*}
where the last equality uses $f_{(R,Y,T)\mid Z=0,V}=f_{(R,Y,T)\mid Z=1,V}$. 

(iii) 
The local average treatment effect under $P_{f_1,f_0}^{\ast}$ is 
\begin{eqnarray*}
E_{P_{f_1,f_0}^{\ast}}[Y_1-Y_0\mid T_0^\ast<T_1^\ast]
&=&
E_{P_{f_1,f_0}^{\ast}}[Y_1\mid T_0^\ast<T_1^\ast]-E_{P_{f_1,f_0}^{\ast}}[Y_0\mid T_0^\ast<T_1^\ast]\\
&=&
\int yf_1(y)d\mu_Y(y)-\int yf_0(y)d\mu_Y(y).
\end{eqnarray*}

\subsubsection*{Case 2: $E_P[TV_{(R,Y,T)\mid V}]>0$.}
Lemma \ref{lemma2} is modified into the framework of Theorem \ref{intergreated_thereom1}. 
\begin{lemma}\label{lemma4}
Under Assumption \ref{assumption8}, $E_P[TV_{(R,Y,T)\mid V}]\leq P^{\ast}(T_0^\ast<T_1^\ast)$.
\end{lemma}
\begin{proof}
The proof is the same as Lemma \ref{lemma2} and this lemma follows from 
$f_{(R,Y,T)\mid Z=1,V}-f_{(R,Y,T)\mid Z=0,V}
=
P^{\ast}(T_0^\ast<T_1^\ast\mid V)(f_{(R_1,Y_1,T_1)\mid T_0^\ast<T_1^\ast,V}-f_{(R_0,Y_0,T_0)\mid T_0^\ast<T_1^\ast,V})$.
\end{proof}

From Lemmas \ref{lemma1} and \ref{lemma4} and Equation \ref{eqIA},  all the three inequalities in Theorem \ref{theorem4} are satisfied for the true value of the local average treatment effect.
To complete Theorem \ref{intergreated_thereom1}, it suffices to show that, for any data generating process $P$, any point $\theta$ satisfying the three inequalities in Theorem \ref{intergreated_thereom1} is the local average treatment effect under some data generating process whose data distribution is equal to $P$. 

Define the two data generating processes: $P_L^{\ast}$ and $P_U^{\ast}$. 
First, $P_L^{\ast}$ is defined by 
\begin{eqnarray*}
&(Z,V)&\sim\ f_{(Z,V)}\\
&(T^\ast_{0},T^\ast_{1})\mid Z,V&=\ (0,1)\\
&(R_0,Y_0,T_0)\mid (T^\ast_{0},T^\ast_{1},Z,V)&\sim\ f_{(R,Y,T)\mid Z=0,V}\\
&(R_1,Y_1,T_1)\mid (T^\ast_{0},T^\ast_{1},Z,V)&\sim\ f_{(R,Y,T)\mid Z=1,V}.
\end{eqnarray*}
Second, $P_U^{\ast}$ is defined as follows. 
Using $\mathrm{sgn}(x)=\mathbbm{1}\{x\geq 0\}-\mathbbm{1}\{x<0\}$, 
define 
$$
H=.5\times\mathrm{sgn}\left(\Delta f_{(R,Y,T)\mid Z,V}(R,Y,T)\right), 
$$
and define $P_U^{\ast}$ as 
\begin{eqnarray*}
&(Z,V)&\sim\ f_{(Z,V)}\\
&(T^\ast_{0},T^\ast_{1})\mid Z,V&=\
\begin{cases}
(0,1)&\mbox{ with probability }\Delta E_P[H\mid Z,V]\\  
(0,0)&\mbox{ with probability }P(H=-.5\mid Z=1,V)\\
(1,1)&\mbox{ with probability }P(H=.5\mid Z=0,V) 
\end{cases}\\
&(R_1,Y_1,T_1)\mid (T^\ast_{0},T^\ast_{1},Z,V)&\sim\ 
\begin{cases}
\frac{\Delta f_{(R,Y,T,H)\mid Z,V}(r,y,t,.5)}{\Delta E_P[H\mid Z,V]}&\mbox{ if }T_0^\ast<T_1^\ast\\
f_{(R,Y,T)\mid H=.5,Z=0,V}(r,y,t)&\mbox{ if }T^\ast_{0}=T^\ast_{1}
\end{cases}\\
&(R_0,Y_0,T_0)\mid (T^\ast_{0},T^\ast_{1},Z,V)&\sim\
\begin{cases} 
-\frac{\Delta f_{(R,Y,T,H)\mid Z,V}(r,y,t,-.5)}{\Delta E_P[H\mid Z,V]}&\mbox{ if }T_0^\ast<T_1^\ast\\
f_{(R,Y,T)\mid H=-.5,Z=1,V}(y,t)&\mbox{ if }T^\ast_{0}=T^\ast_{1}.
\end{cases}
\end{eqnarray*}

\begin{lemma}\label{lemma6}
If $E_P[TV_{(R,Y,T)\mid V}]>0$, then 
(i) $P_L^{\ast}$ generates the data distribution $P$ and the local average treatment effect under $P_L^{\ast}$ is $E_P[\Delta E_P[Y\mid Z,V]]$; and 
(ii) $P_U^{\ast}$ generates the data distribution $P$ and the local average treatment effect under $P_U^{\ast}$ is $E_P[\Delta E_P[Y\mid Z]]/E_P[TV_{(R,Y,T)\mid V}]$.
\end{lemma}
\begin{proof}
(i)
Denote by $f^{\ast}$ the density function of $P_L^{\ast}$.
The data generating process $P_L^{\ast}$ generates the data distribution $P$: 
\begin{eqnarray*}
f^{\ast}_{(R,Y,T)\mid Z=0,V}(r,y,t)
&=&
f^{\ast}_{(R,Y,T)\mid T^\ast=0,Z=0,V}(r,y,t)\\
&=&
f^{\ast}_{(R_0,Y_0,T_0)\mid T^\ast=0,Z=0,V}(r,y,t)\\
&=&
f_{(R,Y,T)\mid Z=0,V}(r,y,t)\\
f^{\ast}_{(R,Y,T)\mid Z=1,V}(r,y,t)
&=&
f^{\ast}_{(R,Y,T)\mid T^\ast=1,Z=1,V}(r,y,t)\\
&=&
f^{\ast}_{(R_1,Y_1,T_1)\mid T^\ast=1,Z=1,V}(r,y,t)\\
&=&
f_{(R,Y,T)\mid Z=1,V}(r,y,t)
\end{eqnarray*}
where the first equality uses $T^\ast=Z$. 
Under $P_L^{\ast}$, the local average treatment effect is $E_P[\Delta E_P[Y\mid Z,V]]$: 
\begin{eqnarray*}
E_{P_L^{\ast}}[Y_1-Y_0\mid T_0^\ast<T_1^\ast]
&=&
E_{P_L^{\ast}}[Y_1]-E_{P_L^{\ast}}[Y_0]\\
&=&
E_P[E_P[Y\mid Z=1,V]]-E_P[E_P[Y\mid Z=0,V]]\\
&=&
E_P[\Delta E_P[Y\mid Z,V]].
\end{eqnarray*}
(ii) 
Note that 
\begin{eqnarray*}
\Delta E_P[H\mid Z,V]
&=&
\sum_{t=0,1}\iint\frac{\mathrm{sgn}\left(\Delta f_{(R,Y,T)\mid Z,V}(r,y,t)\right)\Delta f_{(R,Y,T)\mid Z,V}(r,y,t)}{2}
d\mu_Y(y)d\mu_R(r)\\
&=&
\sum_{t=0,1}\iint\frac{|\Delta f_{(R,Y,T)\mid Z,V}(r,y,t)|}{2}d\mu_Y(y)d\mu_R(r)\\
&=&
TV_{(R,Y,T)\mid V}.
\end{eqnarray*}
Denote by $f^{\ast}$ the density function of $P_U^{\ast}$.
Notice that $f^{\ast}_{(R_{t^\ast},Y_{t^\ast},T_{t^\ast})\mid (T^\ast_{0},T^\ast_{1},Z,V)}$ is well-defined because  
\begin{eqnarray*}
&&\Delta E_P[H\mid Z,V]=TV_{(R,Y,T)\mid V}>0\\
&&\Delta f_{(R,Y,T,H)\mid Z,V}(r,y,t,.5)=\Delta f_{(R,Y,T)\mid Z,V}(r,y,t)\mathbbm{1}\{\Delta f_{(R,Y,T)\mid Z,V}(r,y,t)\geq 0\}\geq 0\\
&&\Delta f_{(R,Y,T,H)\mid Z,V}(r,y,t,-.5)=\Delta f_{(R,Y,T)\mid Z,V}(r,y,t)\mathbbm{1}\{\Delta f_{(R,Y,T)\mid Z,V}(r,y,t)<0\}<0.
\end{eqnarray*}
$P_U^{\ast}$ generates the data distribution $P$: 
\begin{eqnarray*}
f^{\ast}_{(R,Y,T)\mid Z=0,V}(r,y,t)
&=&
P_U^{\ast}(T_0^\ast<T_1^\ast\mid Z=0,V)f^{\ast}_{(R,Y,T)\mid T_0^\ast<T_1^\ast,Z=0,V}(r,y,t)\\
&&+P_U^{\ast}(T_1^\ast=T_0^\ast=0\mid Z=0,V)f^{\ast}_{(R,Y,T)\mid T_1^\ast=T_0^\ast=0,Z=0,V}(r,y,t)\\
&&+P_U^{\ast}(T_1^\ast=T_0^\ast=1\mid Z=0,V)f^{\ast}_{(R,Y,T)\mid T_1^\ast=T_0^\ast=1,Z=0,V}(r,y,t)\\
&=&
P_U^{\ast}(T_0^\ast<T_1^\ast\mid V)f^{\ast}_{(R_0,Y_0,T_0)\mid T_0^\ast<T_1^\ast,V}(r,y,t)\\
&&+P_U^{\ast}(T_1^\ast=T_0^\ast=0\mid V)f^{\ast}_{(R_0,Y_0,T_0)\mid T_1^\ast=T_0^\ast=0,V}(r,y,t)\\
&&+P_U^{\ast}(T_1^\ast=T_0^\ast=1\mid V)f^{\ast}_{(R_1,Y_1,T_1)\mid T_1^\ast=T_0^\ast=1,V}(r,y,t)\\
&=&
-\Delta E_P[H\mid Z,V]\frac{\Delta f_{(R,Y,T,H)\mid Z,V}(r,y,t,-.5)}{\Delta E_P[H\mid Z,V]}\\
&&+P(H=-.5\mid Z=1,V)f_{(R,Y,T)\mid H=-.5,Z=1,V}(r,y,t)\\
&&+P(H=.5\mid Z=0,V)f_{(R,Y,T)\mid H=.5,Z=0,V}(r,y,t)\\
&=&
-(f_{(R,Y,T,H)\mid Z=1,V}(r,y,t,-.5)-f_{(R,Y,T,H)\mid Z=0,V}(r,y,t,-.5))\\
&&+f_{(R,Y,T,H)\mid Z=1,V}(r,y,t,-.5)\\
&&+f_{(R,Y,T,H)\mid Z=0,V}(r,y,t,.5)\\
&=&
f_{(R,Y,T)\mid Z=0,V}(r,y,t),
\end{eqnarray*}
and similarly $f^{\ast}_{(R,Y,T)\mid Z=1,V}(r,y,t)=f_{(R,Y,T)\mid Z=1,V}(r,y,t)$.
Under $P_U^{\ast}$, the local average treatment effect is 
\begin{eqnarray*}
\theta(P_U^{\ast})
&=&
E_{P_U^{\ast}}[Y_1-Y_0\mid T_0^\ast<T_1^\ast]\\
&=&
E_{P_U^{\ast}}[E_{P_U^{\ast}}[Y_1\mid T_0^\ast<T_1^\ast,V]\mid T_0^\ast<T_1^\ast]\\&&-E_{P_U^{\ast}}[E_{P_U^{\ast}}[Y_0\mid T_0^\ast<T_1^\ast,V]\mid T_0^\ast<T_1^\ast]\\
&=&
E_{P_U^{\ast}}[\sum_{t=0,1}\iint y\frac{\Delta f_{(R,Y,T,H)\mid Z,V}(r,y,t,.5)}{\Delta E_P[H\mid Z,V]}d\mu_Y(y)d\mu_R(r)\mid T_0^\ast<T_1^\ast]\\&&+E_{P_U^{\ast}}[\sum_{t=0,1}\iint y\frac{\Delta f_{(R,Y,T,H)\mid Z,V}(r,y,t,-.5)}{\Delta E_P[H\mid Z,V]}d\mu_Y(y)d\mu_R(r)\mid T_0^\ast<T_1^\ast]\\
&=&
E_{P_U^{\ast}}[\sum_{t=0,1}\iint y\frac{\Delta f_{(R,Y,T)\mid Z,V}(r,y,t)}{\Delta E_P[H\mid Z,V]}d\mu_Y(y)d\mu_R(r)\mid T_0^\ast<T_1^\ast]\\
&=&
E_{P_U^{\ast}}[\sum_{t=0,1}\iint y\frac{\Delta f_{(R,Y,T)\mid Z,V}(r,y,t)}{\Delta E_P[H\mid Z,V]}d\mu_Y(y)d\mu_R(r)\frac{P_U^{\ast}(T_0^\ast<T_1^\ast\mid V)}{E_{P_U^{\ast}}[P_U^{\ast}(T_0^\ast<T_1^\ast\mid V)]}]\\
&=&
E_{P_U^{\ast}}[\sum_{t=0,1}\iint y\frac{\Delta f_{(R,Y,T)\mid Z,V}(r,y,t)}{E_{P_U^{\ast}}[P_U^{\ast}(T_0^\ast<T_1^\ast\mid V)]}d\mu_Y(y)d\mu_R(r)]\\
&=&
\frac{E_{P_U^{\ast}}[\sum_{t=0,1}\iint y\Delta f_{(R,Y,T)\mid Z,V}(r,y,t)d\mu_Y(y)d\mu_R(r)]}{E_{P_U^{\ast}}[P_U^{\ast}(T_0^\ast<T_1^\ast\mid V)]}\\
&=&
\frac{E_P[\Delta E_P[Y\mid Z,V]]}{E_P[\Delta E_P[H\mid Z,V]]}\\
&=&
\frac{E_P[\Delta E_P[Y\mid Z,V]]}{E_P[TV_{(R,Y,T)\mid V}]}
\end{eqnarray*}
where the fifth equality comes from Bayes' theorem. 
\end{proof}

Theorem \ref{intergreated_thereom1} follows from the next lemma. 
\begin{lemma}
If $E_P[TV_{(R,Y,T)\mid V}]>0$, then, for every $\lambda\in[0,1]$,  
(i) the mixture distribution $\lambda P_L^{\ast}+(1-\lambda)P_U^{\ast}$ satisfies Assumption \ref{assumption1}; 
(ii) $\lambda P_L^{\ast}+(1-\lambda)P_U^{\ast}$ generates the data distribution $P$; 
(iii) under $\lambda P_L^{\ast}+(1-\lambda)P_U^{\ast}$, the local average treatment effect is 
$$
\lambda E_P[\Delta E_P[Y\mid Z,V]]+(1-\lambda)\frac{E_P[\Delta E_P[Y\mid Z,V]]}{E_P[TV_{(R,Y,T)\mid V}]}.
$$
\end{lemma}
\begin{proof}
(i) 
Under both $P_L^{\ast}$ and $P_U^{\ast}$, $Z$ is independent of $(R_{t^\ast},T_{t^\ast},Y_{t^\ast},T^\ast_{0},T^\ast_{1})$ given $V$ for each $t^\ast=0,1$. 
Furthermore, $P_L^{\ast}$ and $P_U^{\ast}$ have the same marginal distribution for $(Z,V)$: $f_{Z,V}$. 
Therefore, the mixture of $P_L^{\ast}$ and $P_U^{\ast}$ also satisfies the independence. 
Since both$P_L^{\ast}$ and $P_U^{\ast}$ satisfy $T^\ast_{1}\geq T^\ast_{0}$ almost surely, so does the mixture. 
(ii) By Lemma \ref{lemma6}, both $P_L^{\ast}$ and $P_U^{\ast}$ generate the data distribution $P$ and so does the mixture. 
(iii) It follows from the last statement in Lemma \ref{lemma6}. 
\end{proof}

\subsection*{Proof of Theorem \ref{theorem1-less}}
The proof of Theorems \ref{theorem1-less} is similar to Theorems \ref{intergreated_thereom1}. 
Only the difference is to change the definition of $P_U^{\ast}$ as follows. 
Define $H=.5\times\mathrm{sgn}\left(\Delta f_{Y\mid Z}(Y)\right)$ and define $P_U^{\ast}$ as 
\begin{eqnarray*}
&Z&\sim\ P(Z=z)\\
&(T^\ast_{0},T^\ast_{1})\mid Z&=\ 
\begin{cases}
(0,1)&\mbox{ with probability }\Delta E_P[H\mid Z]\\  
(0,0)&\mbox{ with probability }P(H=-.5\mid Z=1)\\
(1,1)&\mbox{ with probability }P(H=.5\mid Z=0) 
\end{cases}\\
&(Y_1,T_1)\mid (T^\ast_{0},T^\ast_{1},Z=z)&\sim\ 
\begin{cases}
f_{T\mid Y=y,Z=z}(t)\frac{\Delta f_{(Y,H)\mid Z}(y,.5)}{\Delta E_P[H\mid Z]}&\mbox{ if }T_0^\ast<T_1^\ast\\
f_{T\mid Y=y,Z=z}(t)f_{Y\mid H=.5,Z=0}(y)&\mbox{ if }T^\ast_{0}=T^\ast_{1}\\
\end{cases}\\
&(Y_0,T_0)\mid (T^\ast_{0},T^\ast_{1},Z=z)&\sim\ 
\begin{cases} 
-f_{T\mid Y=y,Z=z}(t)\frac{\Delta f_{(Y,H)\mid Z}(y,-.5)}{\Delta E_P[H\mid Z]}&\mbox{ if }T_0^\ast<T_1^\ast\\
f_{T\mid Y=y,Z=z}(t)f_{Y\mid H=-.5,Z=1}(y)&\mbox{ if }T^\ast_{0}=T^\ast_{1}\\
\end{cases}
\end{eqnarray*}

\subsection*{Proof of Lemma \ref{lemma3}}
First, note that  
\begin{equation}
TV_{(Y,T)\mid V}=\sup_{h\in\mathbf{H}}\Delta E_P[h(Y,T,V)\mid Z,V].\label{sup_TV}
\end{equation}
This is verified as follows. 
For every $h\in\mathbf{H}$, I have 
\begin{eqnarray*}
\Delta E_P[h(Y,T,V)\mid Z,V]
&=& 
E_P[h(Y,T,V)\mid Z=1,V]-E_P[h(Y,T,V)\mid Z=0,V]\\
&=& 
\sum_{t=0,1}\int h(y,t,V)(f_{(Y,T)\mid Z=1,V}(y,t)-f_{(Y,T)\mid Z=0,V}(y,t))d\mu_Y(y)\\
&=& 
\sum_{t=0,1}\int h(y,t,V)\Delta f_{(Y,T)\mid Z,V}(y,t)d\mu_Y(y)\\
&\leq& 
\frac{1}{2}\sum_{t=0,1}\int|\Delta f_{(Y,T)\mid Z,V}(y,t)|d\mu_Y(y)\\
&=& 
TV_{(Y,T)\mid V}. 
\end{eqnarray*}
Moreover, the above inequality becomes an equality if $h(y,t,v)=.5$ if $\Delta f_{(Y,T)\mid Z,V=v}(y,t)>0$ and $h(y,t,v)=-.5$ if $\Delta f_{(Y,T)\mid Z,V=v}(y,t)<0$.

\cite{abadie:2003} and \cite{frolich:2007} show that 
\begin{equation}
E_P[\Delta E_P[X\mid Z,V]]=E_P\left[\frac{Z-\pi(V)}{\pi(V)(1-\pi(V))}X\right]\label{Delta_FRAC}
\end{equation}
for any random variable $X$. The proof is as follows. 
\begin{eqnarray*}
E_P\left[\frac{Z-\pi(V)}{\pi(V)(1-\pi(V))}X\right]
&=&
E_P\left[\frac{(1-\pi(V))Z+(Z-1)\pi(V)}{\pi(V)(1-\pi(V))}X\right]\\
&=&
E_P\left[\frac{Z}{\pi(V)}X\right]-E_P\left[\frac{1-Z}{(1-\pi(V))}X\right]\\
&=&
E_P\left[\frac{E_P[ZX\mid V]}{\pi(V)}\right]-E_P\left[\frac{E_P[(1-Z)X\mid V]}{(1-\pi(V))}\right]\\
&=&
E_P\left[\frac{\pi(V)E_P[X\mid Z=1,V]}{\pi(V)}\right]\\&&\quad-E_P\left[\frac{(1-\pi(V))E_P[X\mid Z=0,V]}{(1-\pi(V))}\right]\\
&=&
E_P\left[E_P[X\mid Z=1,V]-E_P[X\mid Z=0,V]\right]\\
&=&
E_P\left[\Delta E_P[X\mid Z,V]\right].
\end{eqnarray*}

By Theorem \ref{intergreated_thereom1} and Equation (\ref{sup_TV}), $\Theta_I(P)$ is characterized by 
\begin{eqnarray*}
&&|\theta|E_P[\Delta E_P[h(Y,T,V)\mid Z,V]]\leq |E_P[\Delta E_P[Y\mid Z,V]]|\mbox{ for every }h\in\mathbf{H}\\
&&\theta E_P[\Delta E_P[Y\mid Z,V]]\geq 0\\
&&|\theta|\geq |E_P[\Delta E_P[Y\mid Z,V]]|.
\end{eqnarray*} 
Since the second condition implies $\mathrm{sgn}(\theta)=\mathrm{sgn}(E_P[\Delta E_P[Y\mid Z,V]])$, 
the above three conditions becomes 
\begin{eqnarray*}
&&|\theta|E_P[\Delta E_P[h(Y,T,V)\mid Z,V]]\leq \mathrm{sgn}(\theta)E_P[\Delta E_P[Y\mid Z,V]]\mbox{ for every }h\in\mathbf{H}\\
&&\mathrm{sgn}(\theta)E_P[\Delta E_P[Y\mid Z,V]]\geq 0\\
&&|\theta|\geq\mathrm{sgn}(\theta)E_P[\Delta E_P[Y\mid Z,V]].
\end{eqnarray*}
By Equation (\ref{Delta_FRAC}), $\Theta_I(P)$ is characterized as in Lemma \ref{lemma3}.

\subsection*{Proof of Lemma \ref{H_example}}
Condition (i) implies Assumption \ref{H_conditions} (i).
Condition (ii) implies Assumption \ref{H_conditions} (iii).
The rest of the proof is going to show Assumption \ref{H_conditions} (ii).
Define 
$$
h_P^{\ast}(y,t,v)=.5\times\mathrm{sgn}\left(\Delta f_{(Y,T)\mid Z,V=v}(y,t)\right).
$$ 
Then $h_P^{\ast}=\argmax_{h\in\mathbf{H}}E_P\left[\frac{Z-\pi(V)}{\pi(V)(1-\pi(V))}h(Y,T,V)\right]$.
Define 
$$
h_{P,n}^{\ast}=\argmax_{h\in\mathbf{H}_n}E_P[\Delta E_P[h(Y,T)\mid Z,V]].
$$
By the H\"older continuity of $f_{(Y,T)\mid Z,V}$, 
$$
\max_{(y,t,v)\in I_{n,k}}\Delta f_{(Y,T)\mid Z,V=v}(y,t)-\min_{(y,t,v)\in I_{n,k}}\Delta f_{(Y,T)\mid Z,V=v}(y,t)\leq 2D_0\left(2\frac{D_1}{K_n}\right)^d.
$$
Define $D_n=2D_0\left(2\frac{D_1}{K_n}\right)^d$.
For $I_{n,k}$ with $\sup_{(y,t,v)\in I_{n,k}}|\Delta f_{(Y,T)\mid Z,V=v}(y,t)|>D_n$, the above inequality implies that  the sign of $\Delta f_{(Y,T)\mid Z,V=v}(y,t)$ is constant on $I_{n,k}$. 
For those $I_{n,k}$, $h^{\ast}=h_n^{\ast}$ on $I_{n,k}$. 
Then, on every $I_{n,k}$, either $h^{\ast}=h_n^{\ast}$ or $|\Delta f_{(Y,T)\mid Z,V=v}(y,t)|\leq D_n$. 
Therefore
$$
h_{P,n}^{\ast}(y,t,v)\Delta f_{(Y,T)\mid Z,V=v}(y,t)\geq h_P^{\ast}(y,t,v)\Delta f_{(Y,T)\mid Z,V=v}(y,t)+0.5\times D_n.
$$
Since 
\begin{eqnarray*}
&&
E_P\left[\frac{Z-\pi(V)}{\pi(V)(1-\pi(V))}h(Y,T,V)\right]\\
&&\qquad=
E_P\left[\Delta E_P\left[h(Y,T,V)\mid Z,V\right]\right]\\
&&\qquad=
\iiint h(y,t,v)\Delta f_{(Y,T)\mid Z,V=v}(y,t)f_V(v)\mu_Y(dy)\mu_T(dt)\mu_V(dv)\\
&&\qquad=
\sum_{k=1}^{K_n}\iiint_{I_{n,k}}h(y,t,v)\Delta f_{(Y,T)\mid Z,V=v}(y,t)f_V(v)\mu_Y(dy)\mu_T(dt)\mu_V(dv),
\end{eqnarray*}
it follows that 
$$
E_P\left[\frac{Z-\pi(V)}{\pi(V)(1-\pi(V))}h_{P,n}^{\ast}(Y,T,V)\right]\geq E_P\left[\frac{Z-\pi(V)}{\pi(V)(1-\pi(V))}h_P^{\ast}(Y,T,V)\right]+D_n.
$$
Since $D_n$ converges to zero uniformly over $P$, Assumption \ref{H_conditions} (ii) holds.

\subsection*{Proof of Theorem \ref{fixed_alternative}}
The following theorem is taken from Corollary 5.1 and Theorem 6.1 in \cite{chernozhukov/chetverikov/kato:2014}.  
\begin{theorem}\label{bootstrap_theorem}
Given $\varepsilon_n>0$ with $\varepsilon_n\rightarrow 0$ and $\varepsilon_n\sqrt{\log p_n}\rightarrow\infty$, denote by $\mathcal{H}_{1,n}$ the set of $(\theta,\pi,P)\in\Theta\times\Pi\times\mathcal{P}_0$ for which $\pi=P(Z\mid V=\cdot)$ and 
\begin{equation}\label{local_alternatives}
\max_{j=1,\ldots,p_n}m_j(\theta)/\sigma_j(\theta)\geq(1+\varepsilon_n)\sqrt{2\log(p_n/\alpha)/n}.
\end{equation}
Under the assumptions in Theorem \ref{fixed_alternative}, 
\begin{eqnarray*}
(i)&&\liminf_{n\rightarrow\infty}\inf_{(\theta,\pi,P)\in\Theta\times\Pi\times\mathcal{P}_0\text{ s.t. } \theta\in\Theta_I(P) \text{ and } \pi=P(Z\mid V)}P(T(\theta,\pi)\leq c(\alpha,\theta,\pi))\geq 1-\alpha\\
(ii)&&\lim_{n\rightarrow\infty}\sup_{(\theta,P)\in\mathcal{H}_{1,n}}P(T(\theta,\pi)\leq c(\alpha,\theta,\pi))=0. 
\end{eqnarray*}
\end{theorem}

Theorem \ref{fixed_alternative} (i) follows from 
\begin{eqnarray*}
P(\theta\not\in\mathcal{C}_{\theta,n}(\alpha+\delta))
&\leq&
P(\theta\not\in\mathcal{C}_{\theta,n}(\alpha+\delta) \& \pi\in\mathcal{C}_{\pi,n}(\delta))+P(\pi\not\in\mathcal{C}_{\pi,n}(\delta))\\
&\leq&
P(T(\theta,\pi)>c(\alpha,\theta,\pi) \& \pi\in\mathcal{C}_{\pi,n}(\delta))+P(\pi\not\in\mathcal{C}_{\pi,n}(\delta))\\
&=&
P(T(\theta,\pi)>c(\alpha,\theta,\pi))+P(\pi\not\in\mathcal{C}_{\pi,n}(\delta))\\
&\leq&
\alpha+\delta,
\end{eqnarray*}
where the last inequality comes from Theorem \ref{bootstrap_theorem} (i) and Assumption \ref{asympt_assumption} (iv).

Theorem \ref{fixed_alternative} (ii) is shown as follows. 
Denote by $D_2$ a constant for which $\sigma_j(\cdot)<D_2$.
Let $(\theta,P)$ be any element of $\Theta\times\mathcal{P}_0$ with $\theta\not\in\Theta_I(P)$.
It suffices to show that (\ref{local_alternatives}) holds for sufficiently large $n$. 
If either (\ref{ID_Cond1}) or (\ref{ID_Cond2}) is violated, then (\ref{local_alternatives}) holds for sufficiently large $n$. 
In the rest of the proof, I focus on the case where (\ref{ID_Cond3}) is violated. 
That is, 
$$
\sup_{h\in\mathbf{H}}E_P\left[\frac{Z-\pi(V)}{\pi(V)(1-\pi(V))}\left(|\theta|h(Y,T,V)\right)-\mathrm{sgn}(\theta)Y\right]>0.
$$
Since $\mathbf{H}_n$ converges to $\mathbf{H}$ in the sense of (\ref{approx_vanish}) and $\frac{Z-\pi(V)}{\pi(V)(1-\pi(V))}|\theta|$ is bounded, it follows that, for sufficiently large $n$, there is $h\in\mathbf{H}_n$ such that 
$$
E_P\left[\frac{Z-\pi(V)}{\pi(V)(1-\pi(V))}\left(|\theta|h(Y,T,V)-\mathrm{sgn}(\theta)Y\right)\right]>0.
$$
Denoted by $\kappa$ the value of the left-hand side in the above inequality. 
For sufficiently large $n$, there is $j=3,\ldots,p_n$ such that $m_j(\theta)>\kappa$. 
For such $j$, $m_j(\theta)/\sigma_j(\theta)>\kappa/D_2$.
Therefore (\ref{local_alternatives}) holds for sufficiently large $n$.

\section*{Appendix C: Additional assumptions on the measurement error}
This section considers widely-used assumptions on the measurement structure: (i) a positive correlation between the measurement and the truth and (ii) the measurement error is independent of the other error in the simultaneous equation  system in (\ref{measurement})-(\ref{treatment_assignment}).
\begin{assumption}\label{assumption-missclass}
(i) $1-P(T_0=1)-P(T_1=0)<1$.
(ii) $T_{t^\ast}$ is independent of $(Y_{t^\ast},T^\ast_{0},T^\ast_{1})$ for each $t^\ast=0,1$.
\end{assumption} 
Assumption \ref{assumption-missclass} (i) has been used in the literature on measurement error, e.g., \citealp{mahajan:2006}; \citealp{lewbel:2007}; and \citealp{hu:2008}.
Assumption \ref{assumption-missclass} (ii) is that the measured treatment $T$ is positively correlated with the true variable $T^\ast$ as in \cite{hausman/abrevaya/schott-morton:1998}.

Unlike Assumption \ref{assumption1} itself, the combination of Assumptions \ref{assumption1} and \ref{assumption-missclass} yields on restrictions on the distribution for the observed variables. 
\begin{lemma}\label{more_micsc}
Suppose that Assumptions \ref{assumption1} and \ref{assumption-missclass} hold, and consider an arbitrary data distribution $P$ of $(Y,T,Z)$. Then 
$$
\Delta f_{(Y,T)\mid Z}(y,1)-\Delta f_{(Y,T)\mid Z}(y,0)\geq 0.
$$
\end{lemma}
\begin{proof}
Assumption \ref{assumption-missclass} (i) implies
\begin{eqnarray*}
f_{(Y,T)\mid Z}(y,t)
&=&
\sum_{t^\ast=0,1}f_{T\mid Y=y,T^\ast=t^\ast,Z}(t)f_{(Y,T^\ast)\mid Z}(y,t^\ast)\\
&=&
\sum_{t^\ast=0,1}f_{T_{t^\ast}\mid Y_{t^\ast}=y,T^\ast=t^\ast,Z}(t)f_{(Y,T^\ast)\mid Z}(y,t^\ast)\\
&=&
\sum_{t^\ast=0,1}f_{T_{t^\ast}}(t)f_{(Y,T^\ast)\mid Z}(y,t^\ast)\\
\Delta f_{(Y,T)\mid Z}(y,1)
&=&
\sum_{t^\ast=0,1}f_{T_{t^\ast}}(1)\Delta f_{(Y,T^\ast)\mid Z}(y,t^\ast)\\
\Delta f_{(Y,T)\mid Z}(y,0)
&=&
\sum_{t^\ast=0,1}f_{T_{t^\ast}}(0)\Delta f_{(Y,T^\ast)\mid Z}(y,t^\ast).
\end{eqnarray*}
\begin{eqnarray*}
\Delta f_{(Y,T)\mid Z}(y,1)-\Delta f_{(Y,T)\mid Z}(y,0)
&=&
\sum_{t^\ast=0,1}(f_{T_{t^\ast}}(1)-f_{T_{t^\ast}}(0))\Delta f_{(Y,T^\ast)\mid Z}(y,t^\ast)\\
&=&
(f_{T_0}(1)-f_{T_0}(0))\Delta f_{(Y,T^\ast)\mid Z}(y,0)\\&&\quad+(f_{T_1}(1)-f_{T_1}(0))\Delta f_{(Y,T^\ast)\mid Z}(y,1)\\
&=&
(1-f_{T_0}(1)-f_{T_1}(0))(\Delta f_{(Y,T^\ast)\mid Z}(y,1)-\Delta f_{(Y,T^\ast)\mid Z}(y,0)).
\end{eqnarray*}
As in \citet[][Proposition 1.1]{kitagawa:2014} and \citet[][Theorem 1]{mourifie/wan:2014}, Assumptions \ref{assumption1} implies the following inequalities 
\begin{eqnarray*}
&&f_{(Y,T^\ast)\mid Z=0}(y,0)\geq f_{(Y,T^\ast)\mid Z=1}(y,0)\geq 0\\
&&0\leq f_{(Y,T^\ast)\mid Z=0}(y,1)\leq f_{(Y,T^\ast)\mid Z=1}(y,1). 
\end{eqnarray*}
Therefore 
$$
\Delta f_{(Y,T)\mid Z}(y,1)-\Delta f_{(Y,T)\mid Z}(y,0)\geq 0.
$$
\end{proof}

The sharp identified set is characterized as follows. 
\begin{theorem}\label{theorem1-misscl}
Suppose that Assumptions \ref{assumption1} and \ref{assumption-missclass} hold, and consider an arbitrary data distribution $P$ of $(Y,T,Z)$. 
Denote $[x]_+=x1\{x\geq 0\}$ and denote $\frac{c}{0}=c\times \infty$.
Then $\Theta_I(P)=\Theta$ if $TV_Y=0$; otherwise 
$$
\Theta_I(P)=\left\{\frac{\Delta E_P[Y\mid Z]}{\Delta E_P[T\mid Z]+\omega_1(1-2\Delta E_P[T\mid Z])}:\ \omega_1\in\Omega_1\right\},
$$
where 
$$
\Omega_1=\left[\sup_{y\in\mathbf{Y}}\left[\frac{\Delta f_{(Y,T)\mid Z}(y,1)}{\Delta f_{(Y,T)\mid Z}(y,1)-\Delta f_{(Y,T)\mid Z}(y,0)}\right]_+
,\inf_{y\in\mathbf{Y}}\left[\frac{f_{(Y,T)\mid Z=0}(y,1)}{f_{(Y,T)\mid Z=0}(y,0)-f_{(Y,T)\mid Z=0}(y,1)}\right]_+\right].
$$
\end{theorem}
\begin{proof}
Define 
\begin{eqnarray*}
\left(
\begin{array}{cc}
a_{00}(y)&a_{01}(y)\\
a_{10}(y)&a_{11}(y)
\end{array}
\right)
&=&
\left(
\begin{array}{cc}
f_{(Y,T)\mid Z=0}(y,0)&f_{(Y,T)\mid Z=1}(y,0)\\
f_{(Y,T)\mid Z=0}(y,1)&f_{(Y,T)\mid Z=1}(y,1)\\
\end{array}
\right)
\\
\left(
\begin{array}{cc}
b_{00}(y)&b_{01}(y)\\
b_{10}(y)&b_{11}(y)
\end{array}
\right)
&=&
\left(
\begin{array}{cc}
f_{(Y,T^\ast)\mid Z=0}(y,0)&f_{(Y,T^\ast)\mid Z=1}(y,0)\\
f_{(Y,T^\ast)\mid Z=0}(y,1)&f_{(Y,T^\ast)\mid Z=1}(y,1)\\
\end{array}
\right),
\end{eqnarray*}
and define $p_0=f_{T\mid T^\ast=0}(1)$ and $p_1=f_{T\mid T^\ast=1}(0)$.
Assumption \ref{assumption-missclass} (i) implies
\begin{eqnarray*}
f_{(Y,T)\mid Z}(y,t)
&=&
\sum_{t^\ast=0,1}f_{T\mid Y=y,T^\ast=t^\ast,Z}(t)f_{(Y,T^\ast)\mid Z}(y,t^\ast)\\
&=&
\sum_{t^\ast=0,1}f_{T_{t^\ast}\mid Y_{t^\ast}=y,T^\ast=t^\ast,Z}(t)f_{(Y,T^\ast)\mid Z}(y,t^\ast)\\
&=&
\sum_{t^\ast=0,1}f_{T_{t^\ast}}(t)f_{(Y,T^\ast)\mid Z}(y,t^\ast),
\end{eqnarray*}
so that 
$$
\left(
\begin{array}{cc}
a_{00}(y)&a_{01}(y)\\
a_{10}(y)&a_{11}(y)
\end{array}
\right)
=
\left(
\begin{array}{cc}
1-p_0&p_1\\
p_0&1-p_1
\end{array}
\right)
\left(
\begin{array}{cc}
b_{00}(y)&b_{01}(y)\\
b_{10}(y)&b_{11}(y)
\end{array}
\right).
$$
Assumption \ref{assumption-missclass} (ii) implies that the matrix $\left(
\begin{array}{cc}
1-p_0&p_1\\
p_0&1-p_1
\end{array}
\right)$ is invertible. 
Thus 
$$
\left(
\begin{array}{cc}
b_{00}(y)&b_{01}(y)\\
b_{10}(y)&b_{11}(y)
\end{array}
\right)
=
\frac{1}{1-p_0-p_1}
\left(
\begin{array}{cc}
(1-p_1)a_{00}(y)-p_0a_{10}(y)&(1-p_1)a_{01}(y)-p_0a_{11}(y)\\
-p_1a_{00}(y)+(1-p_0)a_{10}(y)&-p_1a_{01}(y)+(1-p_0)a_{11}(y)
\end{array}
\right)
$$
and 
\begin{eqnarray*}
\theta(P)
&=&
\frac{\Delta E[Y\mid Z]}{\frac{p_1}{1-p_0-p_1}(1-2\Delta E[T\mid Z])+\Delta E[T\mid Z]},
\end{eqnarray*}
because 
\begin{eqnarray*}
\Delta E[T^\ast\mid Z]
&=&
f_{T^\ast\mid Z=1}(1)-f_{T^\ast\mid Z=0}(1)\\
&=&
\int b_{11}(y)dy-\int b_{10}(y)dy\\
&=&
\frac{p_1\int (a_{00}(y)-a_{01}(y))dy+(1-p_0)\int (a_{11}(y)-a_{10}(y))dy}{1-p_0-p_1}\\
&=&
\frac{p_1\Delta E[1-T\mid Z]+(1-p_0)\Delta E[T\mid Z]}{1-p_0-p_1}\\
&=&
\frac{p_1}{1-p_0-p_1}(1-2\Delta E[T\mid Z])+\Delta E[T\mid Z].
\end{eqnarray*}
Define $\omega_0=\frac{p_0}{1-p_0-p_1}$ and $\omega_1=\frac{p_1}{1-p_0-p_1}$.
In the rest of the proof, I am going to show that the sharp identified set for $(\omega_0,\omega_1)$ is $\Omega_0\times\Omega_1$ where 
$$
\Omega_0=\left[\sup_{y\in\mathbf{Y}}\left[\frac{-\Delta f_{(Y,T)\mid Z}(y,0)}{\Delta f_{(Y,T)\mid Z}(y,1)-\Delta f_{(Y,T)\mid Z}(y,0)}\right]_+
,\inf_{y\in\mathbf{Y}}\left[\frac{f_{(Y,T)\mid Z=1}(y,0)}{f_{(Y,T)\mid Z=1}(y,1)-f_{(Y,T)\mid Z=1}(y,0)}\right]_+\right].
$$

First, I am going to show that the identified set for $(\omega_0,\omega_1)$ is a subset of $\Omega_0\times\Omega_1$.
As in \citet[][Proposition 1.1]{kitagawa:2014} and \citet[][Theorem 1]{mourifie/wan:2014}, Assumptions \ref{assumption1} implies the following inequalities 
\begin{eqnarray*}
&&f_{(Y,T^\ast)\mid Z=0}(y,0)\geq f_{(Y,T^\ast)\mid Z=1}(y,0)\geq 0\\
&&0\leq f_{(Y,T^\ast)\mid Z=0}(y,1)\leq f_{(Y,T^\ast)\mid Z=1}(y,1). 
\end{eqnarray*}
In the notation of this proof, 
\begin{eqnarray*}
&&(1-p_1)a_{00}(y)-p_0a_{10}(y)\geq (1-p_1)a_{01}(y)-p_0a_{11}(y)\geq 0\\
&&0\leq -p_1a_{00}(y)+(1-p_0)a_{10}(y)\leq -p_1a_{01}(y)+(1-p_0)a_{11}(y).
\end{eqnarray*}
By some algebraic operations, 
\begin{eqnarray*}
&&\omega_0(a_{11}(y)-a_{01}(y))\leq a_{01}(y)\\
&&\omega_1(a_{00}(y)-a_{10}(y))\leq a_{10}(y)\\
&&\omega_0(a_{11}(y)-a_{01}(y)-a_{10}(y)+a_{00}(y))\geq -(a_{00}(y)-a_{01}(y))\\
&&\omega_1(a_{11}(y)-a_{01}(y)-a_{10}(y)+a_{00}(y))\geq -(a_{11}(y)-a_{10}(y)).
\end{eqnarray*}
Since $\omega_0,\omega_1\geq 0$, 
\begin{eqnarray*}
&&\omega_0[a_{11}(y)-a_{01}(y)]_+\leq a_{01}(y)\\
&&\omega_1[a_{00}(y)-a_{10}(y)]_+\leq a_{10}(y)\\
&&\omega_0(a_{11}(y)-a_{01}(y)-a_{10}(y)+a_{00}(y))\geq -(a_{00}(y)-a_{01}(y))\\
&&\omega_1(a_{11}(y)-a_{01}(y)-a_{10}(y)+a_{00}(y))\geq -(a_{11}(y)-a_{10}(y)).
\end{eqnarray*}
By Lemma \ref{more_micsc}, $a_{11}(y)-a_{01}(y)-a_{10}(y)+a_{00}(y)\geq 0$ and then 
$$
-\frac{a_{00}(y)-a_{01}(y)}{a_{11}(y)-a_{01}(y)-a_{10}(y)+a_{00}(y)}  
\leq\omega_0\leq
\frac{a_{01}(y)}{[a_{11}(y)-a_{01}(y)]_+}
$$
$$
-\frac{a_{11}(y)-a_{10}(y)}{a_{11}(y)-a_{01}(y)-a_{10}(y)+a_{00}(y)}  
\leq\omega_1\leq
\frac{a_{10}(y)}{[a_{00}(y)-a_{10}(y)]_+}.
$$

Then, I am going to show that $\Omega_0\times\Omega_1$ is included in the identified set for $(\omega_0,\omega_1)$.
Let $(\omega_0,\omega_1)$ be any element of $\Omega_0\times\Omega_1$. 
Define $\tilde{p}_0=\frac{\omega_0}{1+\omega_0+\omega_1}$ and $\tilde{p}_1=\frac{\omega_1}{1+\omega_0+\omega_1}$.
Then define 
\begin{eqnarray*}
\left(
\begin{array}{cc}
\tilde{f}_{(Y,T^\ast)\mid Z=0}(y,0)&\tilde{f}_{(Y,T^\ast)\mid Z=1}(y,0)\\
\tilde{f}_{(Y,T^\ast)\mid Z=0}(y,1)&\tilde{f}_{(Y,T^\ast)\mid Z=1}(y,1)\\
\end{array}
\right)
&=&
\left(
\begin{array}{cc}
1-\tilde{p}_0&\tilde{p}_1\\
\tilde{p}_0&1-\tilde{p}_1
\end{array}
\right)^{-1}
\left(
\begin{array}{cc}
a_{00}(y)&a_{01}(y)\\
a_{10}(y)&a_{11}(y)
\end{array}
\right).
\end{eqnarray*}
By construction, 
\begin{eqnarray*}
&&\tilde{f}_{(Y,T^\ast)\mid Z=0}(y,0)\geq\tilde{f}_{(Y,T^\ast)\mid Z=1}(y,0)\geq 0\\
&&0\leq\tilde{f}_{(Y,T^\ast)\mid Z=0}(y,0)\leq\tilde{f}_{(Y,T^\ast)\mid Z=1}(y,0).
\end{eqnarray*}
This is a sufficient condition for $\tilde{f}_{(Y,T^\ast)\mid Z}$ to be consistent with Assumptions \ref{assumption1}, which is shown in \citet[][Proposition 1.1]{kitagawa:2014} and \citet[][Theorem 1]{mourifie/wan:2014}.
Thus $(\omega_0,\omega_1)$ belongs to the identified set for $(\omega_0,\omega_1)$.
\end{proof}

\section*{Appendix D: Compliers-defiers-for-marginals condition}
This section demonstrates that a variant of Theorem \ref{theorem1} still holds under a weaker condition than the deterministic monotonicity condition in Assumption \ref{assumption1} (ii). 
I consider the following assumption. 
\begin{assumption}\label{assnption_dechaisemartin:2014}
(i) For each $t^\ast=0,1$, $Z$ is independent of $(T_{t^\ast},Y_{t^\ast},T^\ast_{0},T^\ast_{1})$.
(ii) There is a subset $C_F$ of $\{T^\ast_1>T^\ast_0\}$ such that 
\begin{eqnarray*}
&&P(C_F)=P(T^\ast_1<T^\ast_0)\\
&&f_{(Y_0,T_0)\mid C_F}=f_{(Y_0,T_0)\mid T^\ast_1<T^\ast_0}\\
&&f_{(Y_1,T_1)\mid C_F}=f_{(Y_1,T_1)\mid T^\ast_1<T^\ast_0}.
\end{eqnarray*}
(iii) $0<P(Z=1)<1$. 
\end{assumption} 
Assumption \ref{assnption_dechaisemartin:2014} (ii) imposes the compliers-defiers-for-marginals condition \citep{dechaisemartin:2014} on the joint distribution of $(Y_{t^\ast},T_{t^\ast})$. Under this assumption, Theorem 2.1 of \cite{dechaisemartin:2014} shows that 
$$
E[Y_0-Y_1\mid C_V]=\frac{\Delta E[Y\mid Z]}{\Delta E[T^\ast\mid Z]}
$$
where $C_V=\{T^\ast_1>T^\ast_0\}\setminus C_F$.\footnote{In fact he uses a weaker condition than the compliers-defiers-for-marginals condition to establish this equality. I use the compliers-defiers-for-marginals condition here, because it makes the characterization in \ref{theorem1_clement} exactly the same to Theorem 2.1.}

\begin{theorem}\label{theorem1_clement}
Suppose that Assumption \ref{assumption1} holds, and consider an arbitrary data distribution $P$ of $(Y,T,Z)$. 
The identified set $\tilde{\Theta}_I(P)$ for $E[Y_0-Y_1\mid C_V]$ is characterized in the same way as Theorem \ref{theorem1}: 
$\tilde{\Theta}_I(P)=\Theta$ if $TV_{(Y,T)}=0$; otherwise, 
$$
\tilde{\Theta}_I(P)
=
\begin{cases}
\left[\Delta E_P[Y\mid Z],\frac{\Delta E_P[Y\mid Z]}{TV_{(Y,T)}}\right]&\mbox{ if }\Delta E_P[Y\mid Z]>0\\
\{0\}&\mbox{ if }\Delta E_P[Y\mid Z]=0\\
\left[\frac{\Delta E_P[Y\mid Z]}{TV_{(Y,T)}},\Delta E_P[Y\mid Z]\right]&\mbox{ if }\Delta E_P[Y\mid Z]<0.
\end{cases}
$$
\end{theorem}
\begin{proof}
Since Theorem \ref{theorem1} gives the sharp identified set under a stronger assumption of this theorem, 
the identified set $\tilde{\Theta}_I(P)$ in this theorem should be equal to or larger than the set in Theorem \ref{theorem1}. 
As a result, it suffices to show that  $\tilde{\Theta}_I(P)$ is a subset of  the set in Theorem \ref{theorem1}. 
By the Assumption \ref{assnption_dechaisemartin:2014} (ii), 
\begin{eqnarray*}
f_{(Y,T)\mid Z=0}(y,t)
&=&
P(T^\ast_1=T^\ast_0=1\mid Z=0)f_{(Y,T)\mid T^\ast_1=T^\ast_0=1,Z=0}(y,t)\\
&&+P(T^\ast_1=T^\ast_0=0\mid Z=0)f_{(Y,T)\mid T^\ast_1=T^\ast_0=0,Z=0}(y,t)\\
&&+P(C_F\mid Z=0)f_{(Y,T)\mid C_F,Z=0}(y,t)\\
&&+P(C_V\mid Z=0)f_{(Y,T)\mid C_V,Z=0}(y,t)\\
&&+P(T^\ast_1<T^\ast_0\mid Z=0)f_{(Y,T)\mid T^\ast_1<T^\ast_0,Z=0}(y,t)\\
&=&
P(T^\ast_1=T^\ast_0=1)f_{(Y_1,T_1)\mid T^\ast_1=T^\ast_0=1}(y,t)\\
&&+P(T^\ast_1=T^\ast_0=0)f_{(Y_0,T_0)\mid T^\ast_1=T^\ast_0=0}(y,t)\\
&&+P(C_F)f_{(Y_0,T_0)\mid C_F}(y,t)\\
&&+P(C_V)f_{(Y_0,T_0)\mid C_V}(y,t)\\
&&+P(T^\ast_1<T^\ast_0)f_{(Y_1,T_1)\mid T^\ast_1<T^\ast_0}(y,t)
\end{eqnarray*}
\begin{eqnarray*}
f_{(Y,T)\mid Z=1}(y,t)
&=&
P(T^\ast_1=T^\ast_0=1\mid Z=1)f_{(Y,T)\mid T^\ast_1=T^\ast_0=1,Z=1}(y,t)\\
&&+P(T^\ast_1=T^\ast_0=0\mid Z=1)f_{(Y,T)\mid T^\ast_1=T^\ast_0=0,Z=1}(y,t)\\
&&+P(C_F\mid Z=1)f_{(Y,T)\mid C_F,Z=1}(y,t)\\
&&+P(C_V\mid Z=1)f_{(Y,T)\mid C_V,Z=1}(y,t)\\
&&+P(T^\ast_1<T^\ast_0\mid Z=1)f_{(Y,T)\mid T^\ast_1<T^\ast_0,Z=1}(y,t)\\
&=&
P(T^\ast_1=T^\ast_0=1)f_{(Y_1,T_1)\mid T^\ast_1=T^\ast_0=1}(y,t)\\
&&+P(T^\ast_1=T^\ast_0=0)f_{(Y_0,T_0)\mid T^\ast_1=T^\ast_0=0}(y,t)\\
&&+P(C_F)f_{(Y_1,T_1)\mid C_F}(y,t)\\
&&+P(C_V)f_{(Y_1,T_1)\mid C_V}(y,t)\\
&&+P(T^\ast_1<T^\ast_0)f_{(Y_0,T_0)\mid T^\ast_1<T^\ast_0}(y,t)
\end{eqnarray*}
\begin{eqnarray*}
\Delta f_{(Y,T)\mid Z}(y,t)
&=&
P(C_F)(f_{(Y_1,T_1)\mid C_F}(y,t)-f_{(Y_0,T_0)\mid C_F}(y,t))\\
&&+P(C_V)(f_{(Y_1,T_1)\mid C_V}(y,t)-f_{(Y_0,T_0)\mid C_V}(y,t))\\
&&+P(T^\ast_1<T^\ast_0)(f_{(Y_0,T_0)\mid T^\ast_1<T^\ast_0}(y,t)-f_{(Y_1,T_1)\mid T^\ast_1<T^\ast_0}(y,t))\\
&=&
P(C_V)(f_{(Y_1,T_1)\mid C_V}(y,t)-f_{(Y_0,T_0)\mid C_V}(y,t)).
\end{eqnarray*}
Based on the definition of the total variation distance, 
$$
TV_{(Y,T)}=P(C_V)\frac{1}{2}\sum_{t=0,1}\int |f_{(Y_1,T_1)\mid C_V}(y,t)-f_{(Y_0,T_0)\mid C_V}(y,t)|d\mu_Y(y),
$$
and therefore 
$$
TV_{(Y,T)}\leq P(C_V)\leq 1.
$$
Since $P(C_V)=\Delta E[T\mid Z]$, 
$$
TV_{(Y,T)}\leq \Delta E[T\mid Z]\leq 1.
$$
This concludes that $E[Y_0-Y_1\mid C_V]$ is included in 
$$
\begin{cases}
[\Delta E[Y\mid Z],\frac{\Delta E[Y\mid Z]}{TV_{(Y,T)}}]&\mbox{ if }\Delta E[Y\mid Z]> 0\\
\{0\}&\mbox{ if }\Delta E[Y\mid Z]=0\\
[\frac{\Delta E[Y\mid Z]}{TV_{(Y,T)}},\Delta E[Y\mid Z]]&\mbox{ if }\Delta E[Y\mid Z]<0.
\end{cases}
$$
\end{proof}
\singlespacing
\bibliography{/Users/tu10/Dropbox/mybib}\onehalfspacing
\end{document}